\documentclass{amsart}
\RequirePackage[utf8]{inputenc}
\usepackage{amsmath,mathtools,amsthm,amsfonts,amssymb,amscd,amsbsy,hyperref,amsbsy,enumerate}
\usepackage{dsfont}

\newcommand{\tens}[1]{\otimes_{#1}}

\newcommand{\Iso}{\operatorname{Iso}}

\newcommand{\dist}{\operatorname{dist}}
\newcommand{\rk}{\operatorname{rk}}

\newcommand{\Aff}{\operatorname{Aff}}

\newcommand{\N}{\mathds N}
\newcommand{\Z}{\mathds Z}
\newcommand{\R}{\mathds R}

\newcommand{\GL}{\mathsf{GL}}

\renewcommand{\O}{\mathsf O}

\newcommand{\g}{\mathrm g}

\newcommand{\Ker}{\ker}

\renewcommand{\span}{\operatorname{span}}

\allowdisplaybreaks

\usepackage{hyperref}
\hypersetup{
    pdftoolbar=true,
    pdfmenubar=true,
    pdffitwindow=false,
    pdfstartview={FitH},
    pdftitle={Subspace foliations and collapse of closed flat manifolds},
    pdfauthor={Renato G. Bettiol, Andrzej Derdzinski, Roberto Mossa, and Paolo Piccione},
    pdfsubject={},
    pdfkeywords={},
    pdfnewwindow=true,
    colorlinks=true, 
    linkcolor=blue,
    citecolor=blue,
    urlcolor=blue,
}

\newtheorem{theorem}{Theorem}[]
\newtheorem{lemma}[theorem]{Lemma}
\newtheorem{proposition}[theorem]{Proposition}
\newtheorem{corollary}[theorem]{Corollary}
\newtheorem{definition}[theorem]{Definition}

\newtheorem{claim}[theorem]{Claim}

\newtheorem*{theorem*}{Theorem}

\newtheorem{mainthm}{\sc Theorem}

\theoremstyle{remark}

\newtheorem{remark}[theorem]{Remark}

\allowdisplaybreaks
\numberwithin{equation}{section}
\numberwithin{theorem}{section}

\title[Subspace foliations and collapse of closed flat manifolds]{Subspace foliations and collapse\\ of closed flat manifolds}
\author[R. G. Bettiol]{Renato G. Bettiol}
\address{City University of New York (Lehman College) \newline
\indent Department of Mathematics  \newline
\indent 250 Bedford Park Blvd W\newline
\indent Bronx, NY, 10468, USA }
\email{r.bettiol@lehman.cuny.edu}

\author[A. Derdzinski]{Andrzej Derdzinski}
\address{The Ohio State University \newline
\indent Department of Mathematics  \newline
\indent 231 W.~18th Avenue \newline
\indent Columbus, OH 43210, USA}
\email{andrzej@math.ohio-state.edu}

\author[R. Mossa]{Roberto Mossa}
\address{Universidade de S\~ao Paulo \newline
\indent Departamento de Matem\'atica \newline
\indent Rua do Mat\~ao, 1010 \newline
\indent S\~ao Paulo, SP, 05508-090, Brazil}
\email{robertom@ime.usp.br}

\author[P. Piccione]{Paolo Piccione}
\address{Universidade de S\~ao Paulo \newline
\indent Departamento de Matem\'atica \newline
\indent Rua do Mat\~ao, 1010 \newline
\indent S\~ao Paulo, SP, 05508-090, Brazil}
\email{piccione@ime.usp.br}

\date{\today}

\begin{document}
\begin{abstract}
We study relations between certain totally geodesic foliations of a closed flat manifold and its collapsed Gromov--Hausdorff limits.  Our main results explicitly identify such collapsed limits as flat orbifolds, and provide algebraic and geometric criteria to determine whether they are singular.
\end{abstract}
\maketitle

\section{Introduction}

Any sequence of closed flat $n$-manifolds with bounded diameter is (trivially) precompact in Gromov--Hausdorff topology.
Although the limit of such a (possibly collapsing) sequence is known to be a closed flat orbifold~\cite{BetDerPic2017}, aside from low-dimensional cases, there seems to be no general method available to explicitly identify this Gromov--Hausdorff limit, or to determine whether it is smooth. In the present paper, we use  certain naturally occurring Riemannian foliations of closed flat manifolds, called \emph{subspace foliations}, to provide such methods. This answers a broad question of Fukaya~\cite[Problem 11.1]{fukaya} in the special case of flat manifolds.

It is well known that every closed flat $n$-manifold is of the form $M_\pi=\R^n/\pi$, where $\pi\subset\Iso(\R^n)$ is a Bieberbach group, i.e., a torsion-free crystallographic group. By the classical Bieberbach Theorems~\cite{bieberbach}, see also~\cite{charlap, szczepa-book,Wolf_book}, the maximal abelian subgroup $L_\pi\subset \pi$ is a lattice in $\R^n$, and there is a short exact sequence $0\to L_\pi\to \pi\to H_\pi\to 0$, where $H_\pi\subset \O(n)$ is a finite group identified with the holonomy group of $M_\pi$. Remarkably, this orthogonal $H_\pi$-representation on $\R^n$ is always reducible~\cite{HisSzc91}, i.e., admits proper invariant subspaces $W\subset\R^n$.
Every such $H_\pi$-invariant subspace $W\subset \R^n$ induces a \emph{subspace foliation} $\mathcal F_W$ on~$M_\pi$, whose leaves are the totally geodesic submanifolds
\begin{equation}\label{eq:leaves}
\mathcal F_W(u)=P_\pi(W+u), \quad u\in W^\perp,
\end{equation}
where $P_\pi\colon \R^n\to M_\pi$ is the covering map. These leaves are themselves flat manifolds, and are either all compact or noncompact. For instance, if $W$ is a line with irrational slope in $\R^2$, then the corresponding leaves $\mathcal F_W(u)$ are dense in the $2$-torus $\R^2/\Z^2$, a flat manifold with trivial holonomy.
More generally, the leaves \eqref{eq:leaves} are compact if and only if the subspace $W$ is \emph{$L_\pi$-generated}, i.e., $W=\operatorname{span}_\R (W\cap L_\pi)$, see Proposition~\ref{thm:leavescompact}.
Any $H_\pi$-invariant subspace $W\subset\R^n$ has an \emph{$L_\pi$-closure} $\widehat W$, which is the smallest $L_\pi$-generated subspace of $\R^n$ containing~$W$, see Section~\ref{sec:labases} for details.
In the above example on the $2$-torus, $\widehat W=\R^2$. In general, the $L_\pi$-closure $\widehat W$ of any $H_\pi$-invariant subspace $W$ is also $H_\pi$-invariant, and the corresponding subspace foliation $\mathcal F_{\widehat W}$ is the (foliation) closure of the subspace foliation $\mathcal F_W$, as shown in Propositions~\ref{thm:Lclosgroupinv} and \ref{thm:closureleaves}. Since the foliation $\mathcal F_{\widehat W}$ is Riemannian, i.e., its leaves are equidistant, the leaf space $M_\pi/\mathcal F_{\widehat W}$ has a natural metric structure. Moreover, since $\mathcal F_{\widehat W}$ is hyperpolar, i.e., there is a closed flat submanifold that intersects all leaves orthogonally, it follows that $M_\pi/\mathcal F_{\widehat W}$ is a flat orbifold.

In each dimension $n\in\N$, there are only finitely many closed flat $n$-manifolds (by the Bieberbach Theorems), hence, up to discarding finitely many elements, any convergent Gromov--Hausdorff sequence of flat $n$-manifolds consists of a sequence of flat metrics on a fixed closed flat $n$-manifold $M_\pi$.
Moreover, all flat metrics on $M_\pi$ are obtained by rescaling a given flat metric in the directions tangent to each different subspace foliation $\mathcal F_{W_i}$, provided the $H_\pi$-repre\-sentation has no repeated irreducible summands~\cite[Thm.~B]{BetDerPic2017}. By a standard diagonal argument, any such collapsing Gromov--Hausdorff limit is the same as one obtained collapsing along a single (largest) subspace foliation $\mathcal F_{W}$. 
Note that the orthogonal directions can be kept unchanged, up to replacing non-collapsing directions in the sequence with their limits. Thus, with no loss of generality, we may fix an $H_\pi$-invariant subspace $W$, an arbitrary flat metric $g$ on $M_\pi$, and consider the family of flat metrics $g^s_W$, $s>0$, realizing the collapse of $g$ along the subspace foliation defined by $W$, that is,
\begin{equation}\label{eq:gs}
g^s_W=s^2g|_{T\mathcal F_W}\oplus g|_{T\mathcal F_W^\perp}, \quad s>0.
\end{equation}
The resulting collapsed limit as $s\searrow0$ is explicitly identified in our first main result:

\begin{mainthm}\label{mainthm:A}
The Gromov--Hausdorff limit of the collapsing family of flat manifolds $(M_\pi,g^s_W)$ as $s\searrow0$ is the leaf space $M_\pi/\mathcal F_{\widehat W}$, where $\widehat W$ is the $L_\pi$-closure of~$W$. Moreover, $M_\pi/\mathcal F_{\widehat W}$ is a flat orbifold isometric to the orbit space of the action on $\widehat W^\perp\subset\R^n$ of the crystallographic group given by the image of the homomorphism
\begin{equation*}
\pi\ni(A,v)\longmapsto \big(A\vert_{\widehat W^\perp},P_{\widehat W^\perp}(v)\big)\in\Iso\!\big(\widehat W^\perp\big),
\end{equation*}
where $P_{\widehat W^\perp}\colon \R^n\to \widehat W^\perp$ denotes the orthogonal projection, and $(A,v)\cdot x=Ax+v$.
\end{mainthm}

Clearly, Theorem~\ref{mainthm:A} refines our earlier result \cite[Thm.~A]{BetDerPic2017}. Moreover,
it fits the general framework of collapsing manifolds with bounded curvature, whose foundations were laid by Cheeger and Gromov~\cite{cheeger-gromov1,cheeger-gromov2} and Fukaya~\cite{fukaya1,fukaya3,fukaya2}. Indeed, the collapsing family of metrics \eqref{eq:gs} corresponds to an $F$-structure on $M_\pi$. Nevertheless, results from the above references hold in far too great generality to yield an explicit description of this $F$-structure, and of its collapsed limit. Meanwhile, specializing only to flat manifolds, 
it becomes possible to precisely identify these objects and describe them algebraically in terms of the subspace foliation~$\mathcal F_{\widehat W}$, as above.
In addition, Theorem~\ref{mainthm:A} sheds light on the inverse problem of \emph{flat desingularization}, i.e., that of constructing a collapsing sequence of closed flat manifolds that converges to a prescribed closed flat orbifold.

In light of Theorem~\ref{mainthm:A}, we shall henceforth assume (without loss of generality) that the $H_\pi$-invariant subspace $W\subset\R^n$ is $L_\pi$-generated, up to replacing it with its $L_\pi$-closure $\widehat W$.
Our next main result provides both geometric and algebraic criteria to determine whether collapsing $M_\pi$ along a subspace foliation produces a singular limit space:

\begin{mainthm}\label{mainthm:B}
Let $M_\pi$ be a closed flat manifold, and $W\subset\mathds R^n$ be an $H_\pi$-invariant and $L_\pi$-generated subspace. The following are equivalent:
\begin{enumerate}[\rm (i)]
\item $M_\pi/\mathcal F_W$ is a smooth closed flat manifold, and $M_\pi\to M_\pi/\mathcal F_W$ is a fiber bundle;
\item All leaves of the subspace foliation $\mathcal F_W$ are isometric;
\item The subspace foliation $\mathcal F_W$ contains no exceptional leaves;
\item $P_{W^\perp}(v)\not\in\mathrm{Im}(A-\mathrm{Id})$ for all $(A,v)\in\pi$ with $A\vert_{W^\perp}\ne\mathrm{Id}$.
\end{enumerate}
\end{mainthm}

The algebraic smoothness criterion given by the equivalence between (i) and (iv) answers a question in~\cite[p.~1250]{BetDerPic2017}.
In the above, an \emph{exceptional leaf} $\mathcal F_W(u)$ is one whose fundamental group is strictly larger than that of some other leaf $\mathcal F_W(u')$, when seen (injected) inside the ambient fundamental group $\pi$, see Definition~\ref{def:exceptionalleaf} for details.
In the context of subspace foliations, this coincides with the standard definition of exceptional leaf in foliation theory (of having nontrivial \emph{leaf holonomy}, cf.~Remarks~\ref{rem:holonomies} and \ref{rem:samedef}). It should be noted that (i), (ii), and (iii) are known to be equivalent for any (regular) Riemannian foliation with totally geodesic leaves, see e.g.~\cite{molino,radeschi}. However, we include them in Theorem~\ref{mainthm:B}, since we shall supply direct proofs of these equivalences, that are more accessible than and independent of the arguments needed to establish them in full generality.
In addition, we also provide an elementary proof of the fact that if one (and hence all) of the equivalent statements in Theorem~\ref{mainthm:B} does not hold, then the set of points in $M_\pi$ that belong to exceptional leaves of $\mathcal F_W$ is meager, see Proposition~\ref{thm:existenceexcept}.

Another interesting question is determining to how many different collapsed limits can a given flat manifold converge. Since all closed flat manifolds $M_\pi$ admit a pair of strongly transverse nontrivial subspace foliations with compact leaves (see Corollary~\ref{thm:flatfoliationstructure}), a natural strategy is to show that collapsing $M_\pi$ along each of these subspace foliations gives rise to different collapsed limits. 
Indeed, we are able to distinguish these collapsed limits by means of an invariant defined in terms of their rational holonomy representation, see Definition~\ref{thm:defisequence}. In particular, combining this invariant with a recent result of Lutowski~\cite{lutowski18} yields the following:

\begin{mainthm}\label{mainthm:C}
Every odd-dimensional closed flat manifold $M_\pi$ admits (at least) two nontrivial collapsing limits $M_\pi/\mathcal F_{W_1}$ and $M_\pi/\mathcal F_{W_2}$ that are not affinely equivalent.
\end{mainthm}

Aside from its intrinsic geometric relevance, the existence of different collapsed limits of $M_\pi=\R^n/\pi$ enables one to construct different $\pi$-periodic solutions in $\R^n$ to several geometric variational problems. For instance,
this method was used to construct $\pi$-periodic solutions to the Yamabe problem on $S^m\times \R^n$ in \cite{BetPic2016}. 

The paper is organized as follows. In Section~\ref{sec:aux}, we recall basic facts about flat manifolds and flat orbifolds, and prove some auxiliary results. 
Abstract lattice-generated subspaces are studied in Section~\ref{sec:labases}, together with the notion of $L$-closure of a subspace, and their interactions with finite groups of orthogonal transformations.
Some of the results in Sections \ref{sec:aux} and \ref{sec:labases} have also appeared in~\cite[Sec.~4]{derdzinski-piccione}.
Section~\ref{sec:foliation} discusses geometric and algebraic properties of subspace foliations and their leaf spaces.
In Section~\ref{sec:collapse}, we identify the Gromov--Hausdorff limit of a flat manifold as it collapses along a subspace foliation, proving Theorem~\ref{mainthm:A}. Singularities of this collapsed limit and their relation to exceptional leaves are analyzed in Section~\ref{sec:excleaves}, where Theorem~\ref{mainthm:B} is proven. Finally, Section~\ref{sec:existence} contains an abstract criterion for the existence of two distinct collapsed limits, which implies Theorem~\ref{mainthm:C}.

\subsection*{Acknowledgements}
It is our great pleasure to thank Marco Radeschi for useful conversations about foliation theory, and the anonymous referees for their commendable attention to detail and
meticulous suggestions that led to this much improved final version of the paper.

The first-named author was supported by grants from the National Science Foundation (DMS-1904342), PSC-CUNY (Award \# 62074-00 50), and Fapesp (2019/19891-9),
the second- and fourth-named authors were supported by a bilateral grant Fapesp-OSU (2015/50265-6), the third-named author was supported by a grant from Fapesp (2018/08971-9), and the fourth-named author was supported by grants from Fapesp (2016/23746-6, 2019/16286-7).

\section{Preliminaries}\label{sec:aux}
\subsection{Conventions and notations} Throughout this paper, we shall assume:
\begin{enumerate}[\rm (i)]
\item A (full) \emph{lattice} in a finite-dimensional real vector space $V$ is any subgroup $L$ of the additive group of $V$ generated (as a group) by a basis of $V$, which then must also be a $\mathds Z$-basis of $L$. In particular, $L\subset V$ is discrete. If $L'\subset L$ is a subgroup that spans $V$, then $L'$ has finite index in $L$. 
\item Given a subspace $W\subset\mathds R^n$, we denote by $W^\perp$ the orthogonal complement of $W$ relative to the Euclidean inner product, and by $P_W\colon\mathds R^n\to W$ the orthogonal projection onto $W$.
\item We identify elements $(A,v)$ of the affine group $\Aff(\R^n)=\GL(n)\ltimes\mathds R^n$ with the affine isomorphism $\mathds R^n\ni x\mapsto Ax+v\in\mathds R^n$. In particular, given an affine subspace $W+u\subset\mathds R^n$ invariant under the affine map $(A,v)$, we denote by $(A,v)\vert_{W+u}$ the restriction of $(A,v)$ to $W+u$ which \emph{also takes values} in $W+u$.
\end{enumerate}

\subsection{Closed flat manifolds and orbifolds}\label{sub:genflatman}
Denote by $\Aff(\R^n)=\GL(n)\ltimes\R^n$ and $\Iso(\R^n)=\O(n)\ltimes\R^n$ the affine group and the isometry group of $\R^n$, respectively. An \emph{$n$-dimensional crystallographic group} is a  discrete subgroup $\pi$ of $\Iso(\R^n)$ with compact fundamental domain in  $\mathds R^n$, i.e., such that there exists a compact subset of $\mathds R^n$ that intersects every orbit of its action 
\begin{equation}\label{eq:action}
\pi\times\mathds R^n\ni\big((A,v),x\big)\longmapsto Ax+v\in\mathds R^n.
\end{equation}
An \emph{$n$-dimensional Bieberbach group} is a torsion-free $n$-dimensional crystallographic group. Note that a crystallographic group is torsion-free if and only if it acts freely on $\mathds R^n$, see \cite[Thm.~3.1.3]{Wolf_book}. By the Clifford--Klein Theorem, closed $n$-dimensional \emph{flat manifolds} are precisely the orbit spaces $\R^n/\pi$ of the isometric action \eqref{eq:action} of $n$-dimensional Bieberbach groups $\pi$. Similarly, $n$-dimensional compact \emph{flat orbifolds} are precisely the orbit spaces $\R^n/\pi$ of the isometric action \eqref{eq:action} of $n$-dimensional crystallographic groups $\pi$, see e.g.~\cite[p.~1251]{BetDerPic2017}.

As discussed in the Introduction, from the Bieberbach theorems, see e.g.~\cite{BetDerPic2017, charlap, szczepa-book, Wolf_book,bieberbach}, if $\pi\subset\Iso(\mathds R^n)$ is a Bieberbach group, then $\pi$ has a maximal normal abelian subgroup $L_\pi$ of finite index, which is a lattice in $\mathds R^n$, and $0\to L_\pi\to\pi\to H_\pi\to0$ is a short exact sequence. The finite group $H_\pi\subset \O(n)$ is identified with the holonomy group of $M_\pi=\mathds R^n/\pi$, and the inclusion $H_\pi\hookrightarrow \O(n)$ is (identified with) its holonomy representation~\cite[Thm.~3.4.5]{Wolf_book}.
Moreover, $L_\pi$ is $H_\pi$-invariant, since $L_\pi$ is normal in $\pi$. 
It also follows from the Bieberbach Theorems that (the isomorphism class of) the holonomy group of $(M_\pi,g)$ does not depend on the choice of flat metric $g$ on~$M_\pi$.

\begin{remark}\label{thm:remconjugate}
By the Bieberbach theorems, isomorphic crystallographic subgroups $\pi_1,\pi_2\subset\Iso(\mathds R^n)$ are \emph{conjugate} in $\Aff(\R^n)$, i.e., there exists $(B,v)\in\Aff(\R^n)$ such that $(B,v)\pi_1(B^{-1},-B^{-1}v)=\pi_2$. Denoting respectively by $L_{\pi_i}$ and $H_{\pi_i}$, $i=1,2$, the lattice and holonomy of $\pi_i$, we have $L_{\pi_2}=B(L_{\pi_1})$ and $BH_{\pi_1}B^{-1}=H_{\pi_2}$. 
\end{remark}

\subsection{Covering torus}\label{subsec:covtorus}
The quotient $\mathds R^n/L_\pi$, which is an $n$-torus, carries a free isometric $H_\pi$-action, whose quotient map is a $k$-sheeted Riemannian covering map $\mathds R^n/L_\pi\to M_\pi$. 
In order to describe this $H_\pi$-action on $\mathds R^n/L_\pi$ via deck transformations, note that for all $A\in H_\pi$, there exists $v\in\mathds R^n$ such that $(A,v)\in\pi$, and $v$ is unique up to elements of $L_\pi$, so the map
\begin{equation}\label{eq:idcohomol}
H_\pi\ni A\longmapsto \overline v_A\in\mathds R^n/L_\pi
\end{equation}
is well-defined, and $(A,v)\in\pi$ if and only if $v\in\overline v_A$. 
For $A\in H_\pi$, denote by $\overline A\colon\mathds R^n/L_\pi\to\mathds R^n/L_\pi$ the corresponding linear isometry of the torus $\R^n/L_\pi$.
The free isometric action of $H_\pi$ on $\mathds R^n/L_\pi$ is given by:
\begin{equation}\label{eq:holactiontorus}
\phantom{,\quad A\in H_\pi,\ x\in\mathds R^n/L_\pi.}(A,x)\longmapsto \overline Ax+\overline v_A,\quad A\in H_\pi,\ x\in\mathds R^n/L_\pi.
\end{equation}
Moreover, \eqref{eq:idcohomol} satisfies $\overline v_{AB}=\overline A\overline v_B+\overline v_A$ and $\overline v_{A^{-1}}=-\overline A^{-1}\overline v_A$, for all $A,B\in H_\pi$.

\subsection{Holonomy invariant subspaces}
For all $A\in H_\pi$, one has $\Ker(A-\mathrm{Id})\ne\{0\}$. 
Indeed, if $k\in\mathds N$ is the order of $A$ and $(A,v)\in\pi$, then 
\begin{equation*}
(A,v)^k = \big(A^k,(\mathrm{Id}+A+\ldots+A^{k-1})v\big).
\end{equation*}
Since $\pi$ is torsion-free, $u=(\mathrm{Id}+A+\ldots+A^{k-1})v\neq 0$ and clearly $u\in\Ker(A-\mathrm{Id})$.
Moreover, by orthogonality, one has:
\begin{equation}\label{eq:orthimker}
\Ker(A-\mathrm{Id})^\perp=\mathrm{Im}(A-\mathrm{Id}).
\end{equation}
Restricting $\mathrm{Id}+A+\ldots+A^{k-1}$ to each summand in $\R^n=\Ker(A-\mathrm{Id})\oplus \mathrm{Im}(A-\mathrm{Id})$, we see that 
$
\mathrm{Id}+A+\ldots+A^{k-1}=k\, P_{\Ker(A-\mathrm{Id})}.
$
In particular, if $A\in H_\pi$ commutes with every other element of $H_\pi$, then $W=\ker(A-\mathrm{Id})$ is a nontrivial $H_\pi$-invariant subspace of $\R^n$.
Remarkably, an invariant subspace always exists, even if $H_\pi$ has trivial center, due to the following result about Bieberbach groups:

\begin{theorem}[Hiss--Szczepa\'nski~\cite{HisSzc91}]\label{thm:hissszczepa91}
Let $\pi\subset\Iso(\R^n)$, $n\geq2$, be any Bieberbach group.  
The rational holonomy representation of $H_\pi$ is not irreducible.
\end{theorem}

In the above, the \emph{rational holonomy representation} is the $H_\pi$-representation on the rational vector space $L_\pi\tens{\mathds Z}\mathds Q$.
The following generalization of Theorem~\ref{thm:hissszczepa91} has been very recently obtained by Lutowski~\cite{lutowski18}:

\begin{theorem}[Lutowski~\cite{lutowski18}]\label{thm:lutowski}
Let $\pi\subset\Iso(\R^n)$, $n\geq2$, be a Bieberbach group with nontrivial holonomy $H_\pi$.
The rational holonomy representation of $H_\pi$ has at least two inequivalent irreducible subrepresentations.
\end{theorem}

Some geometric consequences of Theorem~\ref{thm:lutowski} are discussed in Section~\ref{sec:existence}.

\subsection{Affine equivalences of compact flat orbifolds}\label{sub:affineequiorb}
Recall that two compact $n$-dimensional flat orbifolds are \emph{affinely equivalent} if the corresponding crystallographic groups are conjugate in $\Aff(\R^n)$.
The following statement, which is useful in the sequel, is a consequence of a more general algebraic result~\cite[Thm.~III.2.2]{charlap}.

\begin{proposition}\label{thm:affequivflatorbifolds}
For $i=1,2$, let $E_i\cong\R^n$ be Euclidean spaces, $\pi^i\subset\Iso(E_i)$ a crystallographic group with associated short exact sequence
\[0\longrightarrow L^i\longrightarrow\pi^i\longrightarrow H^{(i)}\longrightarrow1,\]
where $L^i$ is a lattice in $E_i$, and $H^{(i)}\subset\O(E_i)$.
If the corresponding compact flat orbifolds $\mathcal O_1=E_1/\pi^1$ and $\mathcal O_2=E_2/\pi^2$ are affinely equivalent, then the rational holonomy representations of $\mathcal O_1$ and of $\mathcal O_2$ are equivalent, i.e., there exists an isomorphism of $\mathds Q$-vector spaces $T\colon L^1\otimes\mathds Q\to L^2\otimes\mathds Q$ such that $H^{(2)}=TH^{(1)}T^{-1}$.
\end{proposition}

\begin{proof}
Identify the lattices $L^i$ with subgroups of $\pi^i$. Set $n=\dim E_1=\dim E_2$, choose isometries $I_i\colon E_i\to\mathds R^n$, and set $\widetilde\pi_i=I_i\pi_iI_i^{-1}$, $i=1,2$. The orbifolds $\widetilde{\mathcal{O}_i}:=\mathds R^n/\widetilde\pi_i\cong\mathcal O_i$ are affinely equivalent, i.e., there exists $(B,v)\in\Aff(\R^n)$ such that $(B,v)\widetilde\pi_1(B^{-1},-B^{-1}v)=\widetilde\pi_2$.  The desired map $T$ is induced by the group isomorphism 
$I_2^{-1}BI_1\colon L^1\to L^2$, 
see Remark~\ref{thm:remconjugate}.
\end{proof}

In particular, Proposition~\ref{thm:affequivflatorbifolds} implies that a subspace $V_1\subset L^1\otimes\mathds Q$ is $H^{(1)}$-invariant if and only if $T(V_1)$ is $H^{(2)}$-invariant. Similarly, if $V_1$ is $H^{(1)}$-invariant, then $V_1$ is irreducible if and only if $T(V_1)$ is irreducible, motivating the following:

\begin{definition}\label{thm:defisequence}
Given a completely reducible representation $\rho\colon H\to\GL(V)$ of a group $H$ on a finite-dimensional vector space $V$ (over any field), the \emph{i-sequence} of $\rho$ is the $s$-tuple of non-decreasing positive integers $i_\rho=(n_1,\ldots,n_s)$, 
where $s\ge1$ is the number of distinct irreducible $\rho$-invariant subspaces $V_1,\dots, V_s$,
and $n_i=\dim V_i$ for each $1\leq i\leq s$. The positive integer $s$ is called the \emph{length} of the sequence $i_\rho$.
\end{definition}

Note that there may exist $i\neq j$ such that $V_i\cong V_j$ are isomorphic. Furthermore, if the i-sequence of $\rho$ is $i_\rho=(n_1,\cdots,n_s)$, then clearly $n_1+\cdots+n_s=\dim V$.

By the above, the i-sequence of the rational holonomy is an affine invariant:

\begin{corollary}\label{thm:iseqaffequivorb}
Rational holonomy representations of affinely equivalent compact flat orbifolds have the same i-sequence.
\end{corollary}

Finally, note that Theorem~\ref{thm:hissszczepa91} states that the i-sequence $(n_1,\ldots,n_s)$ of the rational holonomy representation of any closed flat manifold $M_\pi$ has length $s\ge2$. Meanwhile, the i-sequence of the rational holonomy representation of a flat orbifold may have length $s=1$, see \cite[Sec.~5.3]{BetDerPic2017} for examples where $H_\pi$ is irreducible.

\subsection{Closed subgroups of vector spaces}
A closed subgroup of a finite dimensional vector space is the sum of a vector subspace and a discrete sugroup.
For the reader's convenience, we include a precise statement and a short proof of this fact:

\begin{proposition}\label{thm:closedsubgroupvectorspace}
Let $V$ be a finite dimensional real vector space, and let $\Gamma\subset V$ be a closed subgroup of $V$. If $\Gamma_0$ is the connected component of $\Gamma$ containing $0$, then $\Gamma_0$ is a vector subspace of $V$. Given any complement $V'$ of $\Gamma_0$ in $V$, $\Gamma'=V'\cap\Gamma$ is a discrete subgroup of $V'$, and $\Gamma=\Gamma_0+\Gamma'$. 
	If $\Gamma$ spans $V$, then $\Gamma'$ spans $V'$.
\end{proposition}

\begin{proof}
Although the first statement above has a short Lie-theoretic proof, see e.g.~\cite[Prop.~3.1]{BetDerPic2017}, we now provide an elementary and direct argument.
First, observe that if $\Gamma$ is discrete, then $\Gamma$ is generated by an $\mathds R$-linearly independent subset of $V$. In particular, $\Gamma$ is a free abelian finitely generated group of rank $\leq \dim V$.

Now, if $\Gamma$ is not discrete, then $\Gamma$ contains a nonzero vector subspace of $V$. Namely, if $\Gamma$ is not discrete, then $0$ is not isolated in $\Gamma$, and there is a sequence $g_k\in\Gamma\setminus\{0\}$ with $\lim g_k=0$. 
Up to taking subsequences, we may assume that $\lim g_k/\Vert g_k\Vert=v\in V$, with $\Vert v\Vert=1$. 
We claim that $\mathds R\cdot v\subset\Gamma$. Indeed, if $t>0$, set $\alpha_k=t\Vert g_k\Vert^{-1}$, so that $\lim\alpha_k=+\infty$ and $\lim\alpha_k\,g_k=tv$. 
Defining $n_k=\lfloor\alpha_k\rfloor$, we have $n_k>0$ for $k$ large, so that $1\le\alpha_k/n_k\le 1+1/n_k$, and therefore $\lim\alpha_k/n_k=1$. This yields $\lim n_kg_k=\lim\alpha_k\g_k=t\,v$. Since $n_kg_k\in\Gamma$ and $\Gamma$ is closed, it follows that $t\,v\in\Gamma$. Clearly, also $-t\,v\in\Gamma$, i.e., $\mathds R\cdot v\subset\Gamma$.

Since $\Gamma$ is closed under taking sums, we may consider  the \emph{largest} subspace $S$ of $V$ contained in $\Gamma$. 
Note that $\Gamma/S$ is a discrete subgroup of $V/S$. Namely, if $P\colon V\to V/S$ is the quotient map, since $\Gamma$ is a closed $P$-saturated subset of $V$, it follows that $P(\Gamma)=\Gamma/S$ is closed in $V/S$. Moreover, the subgroup $\Gamma/S$ does not contain any nontrivial vector subspace of $V/S$, by the maximality of $S$. As we proved above, $\Gamma/S$ must then be discrete in $V/S$.

Since the quotient map $\Gamma\to\Gamma/S$ is continuous, $S$ is open in~$\Gamma$. Clearly, it is also closed, and therefore $S=\Gamma_0$ is the connected component of $\Gamma$ containing $0$. If $V'$ is a complement of $\Gamma_0$, let $P_{V'}\colon V\to V'$ be the projection corresponding to the direct sum decomposition $V=\Gamma_0\oplus V'$. Thus, 
$\Gamma=\Gamma_0+P_{V'}(\Gamma)$, and, by identifying $V'$ with $V/S$ and using the previous statement, $P_{V'}(\Gamma)$ is a closed and discrete subgroup of $V'$. Clearly, $P_{V'}(\Gamma)=\Gamma\cap V'$. As shown above, $P_{V'}(\Gamma)$ is then the $\mathds Z$-span of a linearly independent subset of $V'$, so the last statement follows.
\end{proof}

\section{Lattice-generated subspaces and lattice-closure}\label{sec:labases}

In this section, we develop some abstract elements in the theory of lattice-generated subspaces, including the construction of the lattice-closure of a subspace.

Denote by $V$ an $n$-dimensional real vector space, and by $L\subset V$ a fixed lattice.

\begin{definition}
A subspace $W\subset V$ is \emph{$L$-generated} if $W\cap L$ spans $W$. 
\end{definition}

If $W$ is $L$-generated, then $L\cap W$ is a lattice in $W$; namely, it is discrete and contains a basis of $W$.
Clearly, the sum of a family of $L$-generated subspaces is also $L$-generated. Less obvious is that the intersection of $L$-generated subspaces is also $L$-generated,
which we prove using the following characterization~\cite[Lemma~4.2]{derdzinski-piccione}:

\begin{proposition}\label{thm:charLgen}
A subspace $W\subset V$ is $L$-generated if and only if its projection onto the quotient torus $V/L$ is closed (equivalently, compact).
\end{proposition}

\begin{proof}
Choose an inner product in $V$ and identify the quotient $V/W$ with $W^\perp$. Clearly, the image of $W$ in $V/L$ is closed if and only if there exists $\varepsilon>0$ such that $\dist(0,W+\ell)<\varepsilon$ for $\ell\in L$ implies $\ell\in W$, i.e., if and only if $P_{W^\perp}(L)$ is discrete. 

Choose a $\mathds Z$-basis $(\ell_1,\ldots,\ell_n)$ of $L$ such that $\span_\mathds R(W\cap L)=\span_\mathds R\{\ell_1,\ldots,\ell_s\}$, with $s\le\dim W$. Then, $P_{W^\perp}(L)$ is freely generated by $P_{W^\perp}(\ell_{s+1}),\ldots,P_{W^\perp}(\ell_n)$. If $P_{W^\perp}(L)$ is discrete, then 
$P_{W^\perp}(\ell_{s+1}),\ldots,P_{W^\perp}(\ell_n)$ are linearly independent, and therefore $n-s\le\dim W^\perp=n-\dim W$, i.e., $s=\dim W$ and $\span_\mathds R(W\cap L)=W$. Conversely, if $s=\dim W$, then since $P_{W^\perp}(\ell_{s+1}),\ldots,P_{W^\perp}(\ell_n)$ generate $W^\perp$, they must be linearly independent, and therefore $P_{W^\perp}(L)$ is discrete.
\end{proof}

Note that, by Proposition~\ref{thm:charLgen}, a subspace $W\subset V$ is $L$-generated if and only if the associated foliation $\mathcal F_W$ as in \eqref{eq:leaves} on the torus $M=V/L$ has compact leaves.
In particular, the following intersection property also holds, see also~\cite[Lemma~4.4]{derdzinski-piccione}:

\begin{corollary}\label{thm:intersectionLgen}
The intersection of a family of  $L$-generated subspaces of $V$  is also $L$-generated. 
\end{corollary}

\begin{proof}
Given $L$-generated subspaces $W_1$ and $W_2$ of $V$, the projections of $W_1$ and $W_2$ onto $V/L$ are compact totally geodesic submanifolds (in fact, tori). The intersection of these projections is a compact subgroup of a torus, whose $0$-connected component is a closed, connected subgroup of a torus, hence a torus $\mathcal T$ itself. The tangent space to $\mathcal T$ at $0$ is the intersection $W_1\cap W_2$, and it follows from Proposition~\ref{thm:charLgen} that $W_1\cap W_2$ is $L$-generated. By induction, one easily obtains that the intersection of a finite family of $L$-generated subspaces of $V$ is also $L$-generated. Finally, given an arbitrary family $\mathfrak W=\{W_\alpha\}_{\alpha\in A}$ of $L$-generated subspaces of $V$, if $\{V_{\alpha_1},\ldots,V_{\alpha_k}\}$ is a finite subfamily of $\mathfrak W$ whose intersection has minimal dimension among all finite subfamilies of $\mathfrak W$, then $\bigcap_{\alpha\in A} W_\alpha=\bigcap_{j=1}^k W_{\alpha_j}$, hence
$\bigcap_{\alpha\in A} W_\alpha$ is $L$-generated.
\end{proof}

\subsection{\texorpdfstring{$L$-closure}{Lattice-closure}}
With the above intersection property at hand, we may define:

\begin{definition}
The \emph{$L$-closure} of a subspace $W\subset V$ is the intersection of all $L$-generated subspaces of $V$ that contain $W$. In other words, the $L$-closure of $W$ is the smallest $L$-generated subspace containing $W$. 
\end{definition}

\subsection{\texorpdfstring{Construction of the $L$-closure}{Construction of the closure}}
We now provide details of an explicit construction of the $L$-closure of a subspace, and describe some of its properties.

\begin{lemma}\label{thm:obs1}
If $G_1$, $G_2$ are free abelian groups, $\varphi\colon G_1\to G_2$ is a surjective homomorphism, and $x_j,y_a\in G_1$ (with indices $j,a$ ranging over finite sets) are  such that $x_j$ form a $\mathds Z$-basis of $\Ker \varphi$ and $\varphi(y_a)$ form a $\mathds Z$-basis of $G_2$, then the family consisting of all $x_j$ and $y_a$ forms a $\mathds Z$-basis of $G_1$.
\end{lemma}

\begin{proof} 
It is easy to see that every $g\in G_1$ can be uniquely expressed as an integer combination of $x_j$ and $y_a$.
\end{proof}

\begin{lemma}\label{thm:directsummand}
If $G\subset V$ is a finitely generated (additive) subgroup, then for any subspace $W\subset V$, the intersection $G\cap W$ is a direct summand subgroup of $G$.
\end{lemma}

\begin{proof}
Since for a finitely generated abelian group $G$ being free is equivalent to being torsion-free, it follows from Lemma~\ref{thm:obs1} that a subgroup $G'\subset G$ is a direct summand of $G$ if and only if the quotient $G/G'$ is torsion-free. This holds, in particular, when $G$ is a finitely generated subgroup of a finite-dimensional real vector space $V$, and when $G'=G\cap W$ for some subspace $W$ of $V$.
\end{proof}

\begin{lemma}\label{thm:lemstarshapneighbasis}
Let $G$ be a finitely generated subgroup of the vector space $V$. If $G$ is dense in $V$, then every neighborhood of $0$ in $V$ contains a $\mathds Z$-basis of $G$.
\end{lemma}

\begin{proof}
We proceed by induction on the rank of $G$, denoted $m=\rk G\ge2$. Note that $m>n=\dim V$, since $G$ is dense in $V$. In particular, if $m=2$ then $n \le1$, and the statement follows trivially. Assume the statement holds for all groups of rank less than $m$. Fix a group $G$ of rank $m$ and an Euclidean norm in $V$. Replace the neighborhood of $0$ by an $\varepsilon$-ball around $0$, and choose a $\mathds Z$-basis 
$e_1,\ldots,e_m$ of $G$ such that $0<|e_1 |<\varepsilon/2$. Note that this $\Z$-basis exists since we may choose $e_1\in G$ with this property and, dividing it by a suitable positive integer, ensure (via Lemma~\ref{thm:directsummand}) that it 
generates a di\-rect-sum\-mand sub\-group of $G$. 
Denote by $P\colon V\to V/\R e_1$ the quo\-tient space projection. The images $P(e_2),\ldots,P(e_m)$ generate a dense sub\-group $G'$ in $V/\mathds Re_1$ of rank less than $m$, and so all elements of some new $\mathds Z$-basis $P(\hat e_2),\ldots,P(\hat e_s)$ of $G'$, with $s\le m$, have norm 
less that $\varepsilon/2$. The desired $\mathds Z$-basis of $G$ consists of $e_1$ and $\hat e_2+k_2 e_1,\ldots,\hat e_s+k_s e_1$ for suitable integers $k_2,\ldots,k_s$. More precisely, we project $\hat e_2,\dots,\hat e_s$ orthogonally onto $e_1^\perp$, obtaining $\hat e_2+r_2e_1,\ldots,\hat e_s+r_s e_1$ with some $r_2,\ldots r_s\in\mathds R$. The desired $k_2,\ldots,k_s\in\mathds Z$ are obtained by choosing any integers satisfying $|k_j-r_j|\le1$ for $2\le j\le s$.
\end{proof}

\begin{lemma}\label{thm:Lclosuredenseprojection}
Let $L\subset V$ be a lattice, and $P\colon V\to V/W$ be the quotient map. Then $W$ is $L$-generated if and only if $P(L)$ is discrete in $V/W$.
\end{lemma}

\begin{proof}
If $W$ is $L$-generated, let $\{\ell_1,\ldots,\ell_n\}$ be a $\mathds Z$-basis of $L$, with $\{\ell_1,\ldots,\ell_k\}$ a basis of $W$. Then, $P(L)$ is discrete if and only if $P(\ell_{k+1}),\ldots,P(\ell_n)\in V/W$ are linearly independent. If $\sum_{j=k+1}^n\alpha_jP(\ell_j)=0$, then $\sum_{j=k+1}^n\alpha_j\ell_j\in W$, hence $\alpha_{k+1}=\ldots=\alpha_n=0$, so $P(L)$ is discrete in $V/W$. The converse is trivial.
\end{proof}

We are now in position to give an explicit construction (and establish further structural properties) of the $L$-closure $\widehat W$ of a subspace $W$. 

\begin{proposition}\label{xtens}
Given a fi\-nite-di\-men\-sion\-al real vector space $V$, a lattice 
$L\subset V$, and a vector sub\-space $W\subset V$, denote by $P\colon V\to V/W$ the quo\-tient space projection.
Let $\mathfrak L$ be the closure in $V/W$ of the image $P(L)$, and $K$ be the connected component of $\mathfrak L$ that contains $0$. Set $\widehat W=P^{-1}(K)$. Then the following hold:
\begin{enumerate}[\rm (a)]
\renewcommand{\theenumi}{\textrm{\alph{enumi}}}
\item\label{itm:a} $K$ and $\widehat W$ are vector sub\-spaces of, respectively,  $V/W$ and $V$;
\item\label{itm:b} $L$ has a $\mathds Z$-basis of the form 
$\{w_j,v_a,u_\lambda\}$, with indices $j,a,\lambda$ ranging over 
finite sets, such that the vectors $w_j$ generate $L\cap W$, 
while $w_j$ and $v_a$ together span~$\widehat W$;
\item\label{itm:c} $\widehat W$ is an $L$-generated sub\-space of $V$, containing  $W$, and spanned by the group $L'=L\cap\widehat W$;
\item\label{itm:d} every $L$-generated sub\-space of $V$ that contains $W$ also contains $\widehat W$;
\item\label{itm:e} $P(\widehat W)=K$, and 
$K\cap P(L)=P(L')$ is a dense subset of 
$K$;
\item\label{itm:f} the inclusions $P(L)\subseteq \mathfrak L\subseteq V/W$ and 
$P(L')\subseteq K$ induce a group iso\-mor\-phism 
$P(L)/P(L')\to \mathfrak L/K$ and an injective
homo\-mor\-phism $\mathfrak L/K\to (V/W)/K$, whose image is a full lattice in the quotient vector space $(V/W)/K$.
\end{enumerate}
Furthermore, $w_j,v_a,u_\lambda$ in {\rm (b)} can be chosen so that $\{v_a\}_a\cup \{w_j\}_j$ is a $\Z$-basis of $L\cap W$, $\{P(v_a)\}_a$ is a $\Z$-basis of $P(L')$, and $\{u_\lambda+P(L')\}_\lambda$ is a $\Z$-basis of $P(L)/P(L')$.
\end{proposition}

\begin{proof}
Part \eqref{itm:a} follows readily from Proposition~\ref{thm:closedsubgroupvectorspace}. The first equality of \eqref{itm:e} is obvious, and yields $P(L')\subseteq K\cap P(L)$. For the reverse inclusion, note that any element of $K\cap P(L)=P(\widehat W)\cap P(L)$ may be expressed as $P(v)=P(u)$ with $v\in\widehat W$ and $u\in L$, so that $w=u-v\in W\subseteq\widehat W$, and 
$P(v+w)=P(u)
\in P(\widehat W)\cap P(L)$. The inclusions in \eqref{itm:f} clearly descend to group homo\-mor\-phisms, both of which are injective as $K\cap P(L)=P(L')$. The quotient 
$\mathfrak L/K$, forming a discrete sub\-group of the vector space 
$(V/W)/K$, is a full lattice. Indeed, it spans $(V/W)/K$, since 
$L$ and $P(L)\subseteq \mathfrak L$ span $V$ and $V/W$, respectively. Surjectivity of $P(L)/P(L')\to \mathfrak L/K$ follows; by the above-men\-tion\-ed discreteness of $\mathfrak L/K$, each coset of $K$ contained in $\mathfrak L$ coincides with the closure of its intersection with $P(L)$, hence the intersection is nonempty. 
In particular, $0+K=K$ is the closure of $K\cap P(L)$.
This completes the proof of \eqref{itm:e} and \eqref{itm:f}.

As a consequence of \eqref{itm:f}, we may choose vectors $u_\lambda\in L$, whose image under the composition of quo\-tient space projections $V\to V/W\to (V/W)/K$, or under $P\colon L\to P(L)$ followed by $P(L)\to P(L)/P(L')$, form any prescribed $\mathds Z$-basis of $\mathfrak L/K$ or, respectively, of 
$P(L)/P(L')$. We also fix $w_j\in L$ and $v_a\in L'$ such that $w_j$, or 
$P(v_a)$, constitute any given $\mathds Z$-basis of 
$L\cap W$ or, respectively, $P(L')$.
Lemma~\ref{thm:obs1} can now be applied first to the quo\-tient-pro\-jec\-tion homo\-mor\-phism $P(L)\to P(L)/P(L')$, and then to $P|_L\colon L\to P(L)$, whose kernel is $L\cap W$. The two successive applications show that $P(v_a),P(u_\lambda)$ and $w_j,v_a,u_\lambda$ are $\mathds Z$-bases of $P(L)$ and $L$. The first equality in \eqref{itm:e} implies that $P$ descends to a linear iso\-mor\-phism $V/\widehat W\to (V/W)/K$ which, when preceded by the quo\-tient-space projection $V\to V/\widehat W$, yields the surjective operator $V\to (V/W)/K$ with the kernel $\widehat W$ sending the vectors $w_j,v_a$ to $0$ (as $P(w_j)=0$, while $P(v_a)$ lie in 
$P(L')\subseteq K$), and $u_\lambda$ to a $\Z$-basis of the full 
lattice $\mathfrak L/K\subseteq (V/W)/K$.
Thus, $w_j$ and $v_a$ span $\widehat W$. This establishes \eqref{itm:b}, \eqref{itm:c} and the final statement in the Proposition. 

Finally, to prove \eqref{itm:d}, consider an $L$-generated sub\-space $\widehat V$ of $V$ containing $W$. According to \cite[Rem~4.10]{derdzinski-piccione}, for some open set $U\subseteq V$ equal to a union of cosets of $\widehat V$  (and hence also of $W$) one has $L\cap U=L\cap\widehat V$. Thus, by \eqref{itm:e} and Lemma~\ref{thm:lemstarshapneighbasis}, the open set $P(U)\subseteq V/W$ contains the $\mathds Z$-basis $P(v_a)$ of $P(L')$, and by the last statement in the Proposition, the vectors $v_a$, along with suitable $w_j\in W\subseteq\widehat V$, together span $\widehat W$. On the other hand, in view of the choice of $U$, all $v_a$ lie in $\widehat V$, so $\widehat V$ contains~$\widehat W$.
\end{proof}

By Proposition~\ref{xtens}  \eqref{itm:c} and \eqref{itm:d}, the subspace $\widehat W$ above is the $L$-closure of $W$.

\begin{remark}\label{thm:remquotientorthogonal}
When $V$ is endowed with an inner product, one can identify the quotient $V/W$ with the orthogonal complement $W^\perp\subset V$, and the quotient projection $P\colon V\to V/W$ with the orthogonal projection $P_{W^\perp}\colon V\to W^\perp$. Under these identifications, the subspace $K$ is the connected component of the closure $\overline{P_{W^\perp}(L)}$ in $W^\perp$ that contains $0$, while $\widehat W$ is given by the direct sum $W\oplus K$, and the quotient space $(V/W)/K$ is identified with the orthogonal complement $\widehat W^\perp$.
\end{remark}

\subsection{\texorpdfstring{Invariance by finite subgroups of $\GL(V)$}{Invariance by finite groups of automorphisms}}\label{sub:Lclosureinvariance}
We now discuss how the $L$-closure of subspaces behaves with respect to invariance under certain group actions.

\begin{proposition}\label{thm:Lclosgroupinv}
If $H\subset\GL(V)$ is a group, $L\subset V$ is an $H$-invariant lattice, and $W\subset V$ is an $H$-invariant subspace, then the $L$-closure of $W$ is $H$-invariant. 
\end{proposition}

\begin{proof}
For every $h\in H$ and $L$-generated subspace $W'\subset V$ that contains $W$, we have that $h(W')$ is $L$-generated because $L$ is $H$-invariant, and contains $W$ since $W$ is $H$-invariant. Thus, the family of all $L$-generated subspaces that contain $W$ is $H$-invariant, though each individual subspace need not be. Therefore, the intersection of all members of the family, which is the $L$-closure of $W$, is also $H$-invariant. 
\end{proof}

\begin{lemma}\label{thm:ZbasisLgen}
If $W\subset V$ is $L$-generated, and $k=\dim W$, then there exists a $\mathds Z$-basis $\{\ell_1,\ldots,\ell_n\}$ of $L$ such that $\{\ell_1,\ldots,\ell_k\}$ is a basis of $W$.
\end{lemma}

\begin{proof}
Since $L/(L\cap W)$ is torsion-free, $L\cap W$ is a direct summand in $L$. Take a $\mathds Z$-basis of $W\cap L$ and complete it to a basis of $L$ by joining it with a $\mathds Z$-basis of a (direct sum) complement of $W\cap L$ in $L$.
\end{proof}

In particular, note that Lemma~\ref{thm:ZbasisLgen} implies that any $L$-generated subspace of $V$ admits a complement which is also $L$-generated. This can be refined as follows, see also~\cite[Thm 4.8]{derdzinski-piccione}:

\begin{proposition}\label{thm:ExistcomplHinvLgen}
Let $H\subset\GL(V)$ be a finite group, and suppose $L$ is $H$-invariant. Given an $L$-generated and $H$-invariant subspace $W\subset V$,  there exists a complement $W'$ of $W$ in $V$ which is $L$-generated and $H$-invariant.
\end{proposition}

\begin{proof}
Consider the rational vector space $V_\mathds Q=L\otimes\mathds Q$, and set $W_\mathds Q=(W\cap L)\otimes\mathds Q$, which is a rational subspace of $V_\mathds Q$. Consider the set $\mathcal S$ of all $\mathds Q$-linear projections $P\colon V_\mathds Q\to W_\mathds Q$. We know that $\mathcal S$ is nonempty from Lemma~\ref{thm:ZbasisLgen}. Moreover, $P\mapsto\Ker P$  is clearly a bijection from $S$ to the set of $L$-generated complements of $W$.  Since $L$ is $H$-invariant, $H$ acts on $V_\mathds Q$. There is an action of $H$ on $\mathcal S$ given by
\[H\times\mathcal S\ni (h,P)\longmapsto P_h\in\mathcal S,\]
where $P_h(x)=h^{-1} P(hx)$, for all $x\in V_\mathds Q$. The average $\overline P=\frac1{\vert H\vert}\sum_{h\in H}P_h$ is easily seen to be an element of $\mathcal S$. Since $\overline P$ is $H$-equivariant, its kernel is $H$-invariant, and this is the desired $H$-invariant and $L$-generated complement of $W$.
\end{proof}

\section{Subspace foliations of flat manifolds}
\label{sec:foliation}

In this section, we study the geometry of subspace foliations $\mathcal F_W$ of a flat manifold $M_\pi=\R^n/\pi$, that is, partitions of $M_\pi$ into the totally geodesic submanifolds $\mathcal F_W(u)=P_\pi(W+u)$, $u\in W^\perp$, where $P_\pi\colon \R^n\to M_\pi$ is the covering map, cf.~\eqref{eq:leaves}.
Note that subspace foliations $\mathcal F_W$ are \emph{hyperpolar}, i.e., there exists a totally geodesic flat submanifold $P_\pi(W^\perp)\subset M_\pi$ that intersects all leaves of $\mathcal F_W$ orthogonally.

\begin{remark}\label{thm:remsameleaf}
It is straightforward to verify that the leaves $\mathcal F_W(u)$ and $\mathcal F_W(u')$ coincide if and only if there exists $(A,v)\in\pi$ with $Au+v-u'\in W$. 
\end{remark}

While the leaves $\mathcal F_W(u)$ of a subspace foliation are indexed with $u\in W^\perp$, we shall abuse notation and also write $\mathcal F_W(u)=P_\pi(W+u)$ for any $u\in \R^n$.

\subsection{Compactness} 
We begin by analyzing whether the leaves of $\mathcal F_W$ are compact.

\begin{proposition}\label{thm:leavescompact}
The leaves of $\mathcal F_W$ are compact if and only if  $W$ is $L_\pi$-generated.
\end{proposition}

\begin{proof}
The projection $P_\pi\colon\mathds R^n\to M_\pi$ factors through the projections $\mathds R^n\to\mathds R^n/L_\pi$ and
$\mathds R^n/L_\pi\to M_\pi$. Thus, it suffices to show that, for all $v_0\in\mathds R^n$, the image of the affine subspace $W+v_0\subset\mathds R^n$ in the quotient $\mathds R^n/L_\pi$ is compact (or, equivalently, closed) if and only if $W$ is spanned by $W\cap L_\pi$. Clearly, it is sufficient to consider the case $v_0=0$; this is precisely the result of Proposition~\ref{thm:charLgen}. 
\end{proof}

A version of the above result (for the covering torus) appears in \cite[Lemma~4.2]{derdzinski-piccione}.

\begin{proposition}\label{thm:closureleaves}
The leaves of the subspace foliation $\mathcal F_{\widehat W}$,
where $\widehat W$ is the $L_\pi$-closure of the $H_\pi$-invariant subspace $W$, are the closures of the leaves of $\mathcal F_W$.
\end{proposition}

\begin{proof}
Clearly, each leaf of $\mathcal F_W$ is contained in a leaf of $\mathcal F_{\widehat W}$, which is closed by Proposition~\ref{thm:leavescompact}. As in the proof of Proposition~\ref{thm:leavescompact}, the result follows if we show that the projection of the affine subspace $W+v_0$ on the torus $\R^n/L_\pi$ is dense in the projection of $\widehat W+v_0$. As before, it suffices to consider $v_0=0$. The closure of the projection of $W$ on $\mathds R^n/L_\pi$ is a closed subgroup of $\R^n/L_\pi$, which hence corresponds to an $L_\pi$-generated subspace $W'\subset\R^n$ that contains $W$. Since the projection of $\widehat W$ is a closed subgroup containing the projection of $W$, we have $W'\subset \widehat W$. On the other hand, $\widehat W$ is the smallest $L_\pi$-generated subspace containing $W$, so $W'=\widehat W$.
\end{proof}

\begin{remark}
In foliation theory, the \emph{closure} $\overline{\mathcal F}$ of a (possibly singular) Riemannian foliation $\mathcal F$ on $M$ is defined as the partition of $M$ into the closures of leaves of $\mathcal F$, and this partition is again a (possibly singular) Riemannian foliation~\cite{molino,alexrad}. Thus, Proposition~\ref{thm:closureleaves}
can be restated as $\overline{\mathcal F_W}=\mathcal F_{\widehat W}$. Note that subspace foliations of flat manifolds are always \emph{regular}, i.e., all of its leaves have the same dimension.
\end{remark}

\begin{remark}\label{thm:remmetricstructleavessoace}
Since the leaves of the subspace foliation $\mathcal F_{\widehat W}$ are compact, of the same dimension, and equidistant, the leaf space $M_\pi/\mathcal F_{\widehat W}$ 
has the metric structure of a compact Riemannian orbifold.
Namely, distances on $M_\pi/\mathcal F_{\widehat W}$ are such that $\widehat W^\perp\ni v\mapsto\mathcal F_{\widehat W}(v)\in M_\pi/\mathcal F_{\widehat W}$ is a local isometry, i.e., a Riemannian covering map.
Furthermore, since $\mathcal F_{\widehat W}$ is hyperpolar, the Riemannian orbifold $M_\pi/\mathcal F_{\widehat W}$ is flat.
\end{remark}

\begin{remark}\label{intor}
Recall from Subsection~\ref{subsec:covtorus} that the projection $P_\pi\colon \R^n\to M_\pi$ factors as $\R^n \to \R^n /L_\pi \to M_\pi$, and the latter projection identifies $M_\pi$ with $(\R^n /L_\pi)/ H_\pi$, cf.~\eqref{eq:holactiontorus}.
Both $W$ and its $L_\pi$-closure $\widehat W$ give rise to subspace foliations on the torus $\R^n/L_\pi$, which we also denote by $\mathcal F_W$ and $\mathcal F_{\widehat W}$, respectively. These subspace foliations of $\R^n/L_\pi$ are invariant under the translational action of $\R^n/L_\pi$ on itself, and the leaves of $\mathcal F_{\widehat W}$ are pairwise isometric tori, see Proposition~\ref{thm:leavescompact} and also \cite[Lemma 4.2]{derdzinski-piccione}. Moreover, their images under the projection $\R^n /L_\pi \to M_\pi$ are precisely the leaves of the subspace foliation $\mathcal F_{\widehat W}$ on $M_\pi$, cf.~\cite[Thm~7.1(ii)]{derdzinski-piccione}.
\end{remark}

As claimed in the Introduction, every closed flat manifold $M_\pi$ of dimension $n\geq2$ admits nontrivial subspace foliations $\mathcal F_W$ with compact leaves, as a consequence of Theorem~\ref{thm:hissszczepa91} and Proposition~\ref{thm:leavescompact}.
More precisely, there is a basis $\{\ell_1,\ldots,\ell_n\}$ of $L_\pi$ and $1\leq k\leq n-1$ such that $\{\ell_1,\ldots,\ell_k\}$ spans an $H_\pi$-invariant subspace $W$. Indeed, by Theorem~\ref{thm:hissszczepa91}, one can find $\{\ell_1,\ldots,\ell_k\}\subset L_\pi$ whose $\mathds Q$-span is $H_\pi$-invari\-ant, so the claim follows from Lemma~\ref{thm:ZbasisLgen}. 
Moreover, Proposition~\ref{thm:ExistcomplHinvLgen} yields an even stronger conclusion, as $W$ has an $H_\pi$-invariant and $L_\pi$-generated complement~$W'$.

\begin{corollary}\label{thm:flatfoliationstructure}
Every closed flat manifold $M_\pi$ admits a pair of nontrivial \emph{strongly transversal} subspace foliations $\mathcal F_W$ and $\mathcal F_{W'}$ with compact leaves, that is, such that for all $p\in M_\pi$, $T_pM_\pi$ is the direct sum of the tangent spaces to the leaves through $p$ of each of these foliations.
\end{corollary}

\subsection{Flat structure of leaves}\label{sub:geomleaves}
Henceforth, up to replacing $W$ by its $L_\pi$-closure, assume that $W\subset\R^n$ is $H_\pi$-invariant and $L_\pi$-generated. In particular, the leaves $\mathcal F_W(u)$, $u\in W^\perp$, are compact and totally geodesic submanifolds of $M_\pi$, and hence closed flat manifolds themselves. Thus, intrinsically, each leaf $\mathcal F_W(u)$ is isometric to $W/\pi_W(u)$, for some Bieberbach group $\pi_W(u)\subset\Iso(W)$, which we now identify.

\begin{proposition}\label{thm:Bieberbachleaf}
For all $u\in\mathds R^n$, the Bieberbach group of $\mathcal F_W(u)$ is isomorphic to the subgroup $G_{W}(u)\subset\pi$ that preserves the affine subspace $W+u$, namely
\begin{equation}\label{eq:defGuW}
G_{W}(u)=\big\{(A,v)\in\pi:(A-\mathrm{Id})u+v\in W\big\}.
\end{equation}
\end{proposition}

\begin{proof}
A straightforward computation shows that \eqref{eq:defGuW} is the subgroup of $\pi$ consisting of elements that preserve $W+u$.
We now argue that $G_W(u)$ is isomorphic to the fundamental group of $\mathcal F_W(u)$. First, note that if $(A,v)$ maps some point in $W+u$ to some other point in $W+u$, then $(A,v)\in G_W(u)$. Namely, since $A$ preserves $W$, $(A,v)$ maps $W+u$ to some affine subspace of $\mathds R^n$ which is parallel to $W$. Thus, $(A,v)$ preserves $W+u$, since two distinct parallel affine subspaces are disjoint.

Clearly, the action of $G_W(u)$ on $W+u$ is properly discontinuous, and restricting the projection $P_\pi$ to $W+u$ gives a continuous surjection $P_\pi\colon(W+u)\to\mathcal F_W(u)$. Two points $w+u,w'+u\in W+u$ have the same image under $P_\pi$ if and only if there is $(A,v)\in\pi$ with $A(w+u)+v=w'+u$, i.e., if and only if $w'=Aw+(A-\mathrm{Id})u+v$. From the above, such $(A,v)$ must belong to $G_W(u)$. Therefore, $P_\pi\colon(W+u)\to\mathcal F_W(u)$ is a covering map and $G_W(u)$ is the group of deck transformations. Since $W+u$ is simply-connected, this shows that the fundamental group of $\mathcal F_W(u)$ is isomorphic to the image of the restriction map:
\begin{equation}\label{eq:restrmap}
G_W(u)\ni (A,v)\longmapsto (A,v)\vert_{W+u}\in\Iso(W+u).
\end{equation}
Since $\pi$ acts without fixed points, \eqref{eq:restrmap} is an injective map, concluding the proof.
\end{proof}

We remark that a version of the above result appears in \cite[Thm 7.1 (ii) (a)]{derdzinski-piccione}.

\begin{corollary}\label{thm:bieberleaves}
The closed flat manifold $\mathcal F_W(u)$, $u\in\mathds R^n$, is isometric to the orbit space $W/\pi_W(u)$ of the Bieberbach group $\pi_W(u)$ on the Euclidean space $W$, where
\begin{equation*}
\pi_W(u)=\Big\{\big(A\vert_W,(A-\mathrm{Id})u+v\big)\in\Iso(W):(A,v)\in G_W(u)\Big\}.
\end{equation*}
\end{corollary}

\begin{proof}
Follows readily using conjugation with the isometry $(\mathrm{Id},u)\colon W\to W+u$.
\end{proof}

We now identify the corresponding lattice $L_W(u)\subset W$, and holonomy group $H_W(u)\subset\O(W)$, such that $0\to L_W(u)\to \pi_W(u)\to H_W(u)\to 0$ is the short exact sequence yielded by the Bieberbach theorems applied to $\mathcal F_W(u)=W/\pi_W(u)$.

\begin{remark}\label{rem:holonomies}
We shall refer to $H_W(u)\subset \O(W)$ as the holonomy group of $\mathcal F_W(u)$, since it is identified with its holonomy group \emph{as a closed flat manifold}.
This is not to be confused with the leaf holonomy group $\mathrm{Hol}_p(\mathcal F_W(u))$, which is generated by parallel transports along loops based at $p\in \mathcal F_W(u)$ of vectors \emph{normal} to $\mathcal F_W(u)$.
More precisely, $\mathrm{Hol}_p(\mathcal F_W(u))$ is the image of $\pi_1(\mathcal F_W(u),p)\cong G_W(u)$ in the group of linear isometries of the normal space $\nu_p(\mathcal F_W(u))\cong W^\perp$, see~\cite{molino,radeschi}.
\end{remark}

From Corollary~\ref{thm:bieberleaves}, it is easy to give an abstract characterization of the holonomy $H_W(u)$ and the lattice $L_W(u)$ of $\pi_W(u)$. More precisely, $H_W(u)$ is the image of the map $G_W(u)\ni(A,v)\mapsto A\vert_W\in\O(W)$, while $L_W(u)=\big\{v\in W:\exists\,(A,v)\in G_W(u),\ \text{with}\ A\vert_W=\mathrm{Id}\big\}$.
In particular, $L_\pi\cap W\subset L_W(u)$ for all $u$; namely, for all $v\in L_\pi\cap W$, $(\mathrm{Id},v)\in G_W(u)$. It also follows that, given $u,u'\in\mathds R^n$,
\begin{equation}\label{eq:inclusionGLH}
G_W(u)\subset G_W(u')\;\Longrightarrow\; H_W(u)\subset H_W(u'),\, \text{and}\ L_W(u)\subset L_W(u'),
\end{equation}
and, if $u,u'\in W^\perp$,
\begin{equation}\label{eq:inclusionuu'}
G_W(u)\subset G_W(u')\;\Longrightarrow\; (A-\mathrm{Id})u=(A-\mathrm{Id})u',\;\text{for all}\,(A,v)\in G_W(u).
\end{equation}

\subsection{Algebraic description of the leaf space}\label{sub:algebraiccollapse}
We now describe the leaf space $M_\pi/\mathcal F_W$ as a compact flat orbifold, i.e., as the orbit space of a crystallographic~group.

\begin{lemma}\label{thm:proWperpLgen}
If $W\subset\mathds R^n$ is $L_\pi$-generated, then $P_{W^\perp}(L_\pi)$ is a lattice in $W^\perp$.
\end{lemma}

\begin{proof}
Choose a basis $\ell_1,\ldots,\ell_n$ as in Lemma~\ref{thm:ZbasisLgen}, so that
\begin{equation*}
P_{W^\perp}(L_\pi)=\span_\mathds Z\big\{P_{W^\perp}(\ell_{k+1}),\ldots,P_{W^\perp}(\ell_{n})\big\}.
\end{equation*}
Since $\span_\mathds R\big\{\ell_{k+1},\ldots,\ell_{n}\big\}$ is a complement of $W$, $\{P_{W^\perp}(\ell_{k+1}),\ldots,P_{W^\perp}(\ell_{n})\big\}$ is a basis of $W^\perp$, which concludes the proof.
\end{proof}

\begin{proposition}\label{thm:piperpcrystallographic}
Let $W\subset\mathds R^n$ be an $H_\pi$-invariant $L_\pi$-generated subspace. Then
\begin{equation}\label{eq:homoperp}
\pi\ni(A,v)\longmapsto\big(A\vert_{W^\perp},P_{W^\perp}(v)\big)\in\Iso(W^\perp)
\end{equation}
is a group homomorphism, and its image is a crystallographic subgroup of $\Iso(W^\perp)$.
\end{proposition}

\begin{proof}
This map is a group homomorphism since $P_{W^\perp}$ commutes with all $A\in H_\pi$. Its image contains the lattice $P_{W^\perp}(L_\pi)$, hence its action on $W^\perp$ is cocompact. To show it is discrete, it suffices to show that $(\mathrm{Id}_{W^\perp},0)$ is isolated. Suppose $(A_k,v_k)\in\pi$ is a sequence such that $\big(A_k\vert_{W^\perp},P_{W^\perp}(v_k)\big)$ converges to $(\mathrm{Id}_{W^\perp},0)$. Since $H_\pi$ is finite, we may assume that $A_k = A$ for all $k$, with $A\in H_\pi$ such that $A\vert_{W^\perp}=\mathrm{Id}_{W^\perp}$. We may also assume that $v_k=v+\ell_k$, where $(A,v)\in\pi$ and $\ell_k\in L_\pi$ for all $k$. Then, $P_{W^\perp}(v_k)=P_{W^\perp}(v)+P_{W^\perp}(\ell_k)$, and the set 
$\big\{P_{W^\perp}(\ell_k):k\in\mathds N\big\}$ is closed in $W^\perp$. It follows that $P_{W^\perp}(v+\ell_k)=0$ for sufficiently large $k$, i.e., the sequence
$\big(A_k\vert_{W^\perp},P_{W^\perp}(v_k)\big)$ eventually becomes constant, so $(\mathrm{Id}_{W^\perp},0)$ is isolated.
\end{proof}

\begin{theorem}\label{thm:crystallographiccollapse}
Let $M_\pi=\R^n/\pi$ be a closed flat manifold, and $W\subset \R^n$ be an $H_\pi$-invariant and $L_\pi$-generated subspace.
The leaf space $M_\pi/\mathcal F_W$ is isometric to the flat orbifold $W^\perp/\pi^\perp$, where $\pi^\perp\subset\Iso(W^\perp)$ is the crystallographic group given by the image of the homomorphism \eqref{eq:homoperp}.
\end{theorem}

\begin{proof}
From Remark~\ref{thm:remsameleaf}, two elements $u,u'\in W^\perp$ define the same leaf if and only if there exists $(A,v)\in\pi$ such that $Au-u'+v\in W$, i.e., $\big(A\vert_{W^\perp},P_{W^\perp}(v)\big)u=u'$. 
Thus, the map $\mathfrak l\colon W^\perp/\pi^\perp\to M_\pi/\mathcal F_W$ that carries the $\pi^\perp$-orbit of $u\in W^\perp$ to the leaf $\mathcal F_W(u)$ is a well-defined bijection, and a local isometry by Remark~\ref{thm:remmetricstructleavessoace}, hence an isometry.
\end{proof}

\begin{corollary}\label{thm:holattcollapse}
The holonomy group $H^\perp\subset\O(W^\perp)$ and lattice $L^\perp\subset W^\perp$ associated to the flat orbifold $M_\pi/\mathcal F_W=W^\perp/\pi^\perp$ by the Bieberbach Theorems are:
\begin{enumerate}[\rm (i)]
\item $H^\perp\subset\O(W^\perp)$ is the image of the map $H_\pi\ni A\mapsto A\vert_{W^\perp}\in\O(W^\perp)$;
\item $L^\perp\subset W^\perp$ is the image of the map $\mathcal L_W\ni(A,v)\mapsto P_{W^\perp}(v)\in W^\perp$, where $\mathcal L_{W}=\big\{(A,v)\in\pi:A\vert_{W^\perp}=\mathrm{Id}_{W^\perp}\!\big\}$. This is a lattice in $W^\perp$ which contains $P_{W^\perp}(L_\pi)$ as a finite index subgroup, and therefore $L^\perp\otimes\mathds Q=P_{W^\perp}(L_\pi)\otimes\mathds Q$.
\end{enumerate}
\end{corollary}

\begin{proof}
The identifications of the holonomy and lattice of $W^\perp/\pi^\perp$ with $H^\perp$ and $L^\perp$ respectively follow from Theorem~\ref{thm:crystallographiccollapse}. Clearly, $L^\perp$ contains $P_{W^\perp}(L_\pi)$, which by Lemma~\ref{thm:proWperpLgen} is also a lattice in $W^\perp$. Thus, $P_{W^\perp}(L_\pi)$ has finite index in $L^\perp$.
\end{proof}

\section{Collapse of flat manifolds}\label{sec:collapse}
In this section, we give a proof of Theorem~\ref{mainthm:A} stated in the Introduction by combining Theorem~\ref{thm:crystallographiccollapse} with an identification of the Gromov--Hausdorff limit of the collapsing sequence of flat manifolds $(M_\pi,g_W^s)$ as $s\searrow0$. 

Recall that, given compact metric spaces $(X,d_X)$ and $(Y,d_Y)$, an \emph{$\varepsilon$-approximation} from $X$ to $Y$ is a map $f\colon X\to Y$ such that $\left\vert d_X(x_1,x_2)-d_Y\big(f(x_1),f(x_2)\big)\right\vert<\varepsilon$ for all $x_1,x_2\in X$, and such that $Y$ is in the $\varepsilon$-neghborhood of $f(X)$. It is well known that a sequence of compact metric spaces $(X_n,d_n)$ converges in Gromov--Hausdorff sense to a compact metric space $(X_\infty,d_\infty)$ if and only if for all $\varepsilon>0$ there exists $N_\varepsilon\in\mathds N$ and  $\varepsilon$-approximations $f_n^\varepsilon\colon X_n\to X_\infty$ and $g_n^\varepsilon\colon X_\infty\to X_n$ for all $n\ge N_\varepsilon$.

\begin{lemma}\label{ghcnv}
Let $\rho^s$, $s\in(0,1]$, be distance functions on $M$,  and let $\varPhi\colon M \to Q$ be a map
onto the metric space $(Q,\delta)$ such that 
\begin{equation*}
\delta(\varPhi(x),\varPhi(y))\le\rho^s(x,y)\le\delta(\varPhi(x),\varPhi(y))+d(s)
\end{equation*}
for all $x,y\in M$, where $(0,1]\ni s\mapsto d(s)\in [0,\infty)$ is a function such that $d(s)\to0$ as $s\searrow0$.  Then the Gro\-mov--Haus\-dorff limit of $(M,\rho^s)$ as $s\searrow0$ is $(Q,\delta)$.
\end{lemma}

\begin{proof}
Given $\varepsilon>0$, a pair of $\varepsilon$-ap\-prox\-i\-ma\-tions between $(M ,\rho^s)$ and $(Q,\delta)$ is provided, when $d(s)<\varepsilon$, by $\varPhi\colon M\to Q $ and any map $\theta\colon Q\to M$ with $\theta\circ\varPhi=\mathrm{Id}_M$. Note that $\theta$ need not be continuous, and exists by the Axiom of Choice.
\end{proof}

The following result is a crucial step in the proof of Theorem~\ref{mainthm:A}.

\begin{theorem}\label{unifo}
Let $W\subset\R^n$ be an $H_\pi$-invariant subspace, $\widehat W$ be its $L_\pi$-closure, and $\widehat{L}$ be the lattice in $\widehat{W}$ given by $\widehat L=L_\pi\cap\widehat W$. Consider the Riemannian metric induced by $\langle\cdot,\cdot\rangle_s=s^2 \langle\cdot,\cdot\rangle|_{W}\oplus \langle\cdot,\cdot\rangle|_{W^\perp}$ on the torus $\widehat{W}/\widehat{L}$, and denote by $\rho^s$ the corresponding distance function.
Then its diameter $d(s)=\operatorname{diam}(\widehat W/\widehat L,\rho^s)$ satisfies $\lim\limits_{s\searrow0} d(s)=0$. Moreover, the limit $\rho^0$ of these distance functions vanishes identically.
\end{theorem}

\begin{proof}
For each $s\in(0,1]$, we have the Euclidean norm $|\cdot|_s$ on $\widehat{W}$ defined by $|w+w'|_s^2 =s^2|w|^2 +|w'|^2$ for all $w\in W $ and $w'\in W^\perp$, where $W^\perp$ is the orthogonal complement of $W$ in $\widehat{W}$. 
Since points in $\widehat{W}/\widehat{L}$ are cosets of $\widehat{L}$ in $\widehat{W}$ and their $\rho^s$-dis\-tance is the $|\cdot|_s$-dis\-tance between the corresponding cosets, our assertion follows if we establish the existence, for any $\varepsilon>0$, of some $s_\varepsilon\in\left(0,1 \right]$ satisfying:
\begin{equation}\label{exi}
\mathrm{for\ every\ } \widehat w\in\widehat W \mathrm{\ and\ } s\in\left(0,s _\varepsilon\right], \mathrm{\ there\ exists\ }\widehat\lambda\in\widehat L \mathrm{\ with\ }|\widehat w-\widehat\lambda|_s<\varepsilon.
\end{equation}
Note that, by homogeneity, we may assume that one of the two cosets is $\widehat L$ itself.

To prove the above, identify $\widehat{W}/W$ with the orthogonal complement $W^\perp$. 
By Proposition~\ref{xtens}(e), see also Remark~\ref{thm:remquotientorthogonal}, 
$P_{W^\perp}(\widehat L)$ is a dense additive sub\-group of $W^\perp$.
Let $K\subseteq\widehat W $ be a fixed compact fundamental domain for the translational action of $\widehat L$. Density of $P_{W^\perp}(\widehat L)$ in $W^\perp$ and compactness of $P_{W^\perp}(K)$ allow us to choose an integer $m\ge1$, points $w_1,\dots,w_m\in P_{W^\perp}(K)$, and $\lambda_1,\dots,\lambda_m \in\widehat L$ such that each $P_{W^\perp}(\lambda_i)$, $i\in\{1,\dots,m\}$, lies in the open ball in $W^\perp$ centered at $P_{W^\perp}(w_i)$ of radius $\varepsilon/4$, while the union of these $m$ open balls contains $P_{W^\perp}(K)$. Let $R/2$ be the radius of an open ball in $\widehat W$ centered at $0$  containing $K\cup\{\lambda_1,\dots,\lambda_m\}$. Then \eqref{exi} holds if we define $s_\varepsilon$ by $2Rs _\varepsilon=\sqrt{3}\varepsilon$. Namely, fix $\widehat w\in\widehat W$. Since $K$ is a fun\-da\-men\-tal do\-main, we may fix $\lambda_0\in\widehat L$ such that $\widehat w-\lambda_0\in K$. Generally, whenever $w'\in K$, the open ball in $W^\perp$ centered at $P_{W^\perp}(w')$ with radius $\varepsilon/2$ contains one of the $m$ open balls radius $\varepsilon/4$ (that to which $P_{W^\perp}(w')$ belongs) and, along with it, one of $P_{W^\perp}(\lambda_i)$, $i=1,\dots,m$. Applied to $w' =\widehat w-\lambda_0$, this yields the existence of $i\in\{1,\dots,m\}$ with $| P_{W^\perp}(\widehat w-\widehat\lambda)|<\varepsilon/2$, where $\widehat\lambda=\lambda_0+\lambda_i\in\widehat L$. Our choice of $R$ makes the norms of both $\widehat w-\widehat\lambda=(\widehat w-\lambda_0)-\lambda_i$ and its $W$-component less than $R$, and so $|\widehat w-\widehat\lambda|_s^2 <(sR)^2 +\varepsilon^2/4$, while $(sR)^2 +\varepsilon^2/4\le\varepsilon^2$ when $s\in\left(\smash{0,s _\varepsilon}\right]$.
	
Finally, $\rho^0\equiv0$. Namely, Proposition~\ref{thm:closureleaves} implies that the leaves of the subspace foliation $\mathcal F_W$ are dense in the torus $\widehat{W} /\widehat{L}$. If $x,y\in\widehat{W} /\widehat{L}$ and $\varepsilon>0$, let $y'$ in the leaf through $x$ be such that $\rho^1(y,y')<\varepsilon$. Since $\rho^0(x,y')=0$, and $\rho^0\le\rho^1$, the triangle inequality for $\rho^0$ implies $\rho^0(x,y)<\varepsilon$, concluding the proof.
\end{proof}

\begin{proof}[Proof of Theorem~\ref{mainthm:A}]
For all $s>0$, denote by $\rho^s\colon M_\pi\times M_\pi\to\R$ the distance function on $M_\pi$ induced by the Riemannian metric $g_W^s$ as in \eqref{eq:gs}. Similarly, replacing $W$ by its $L_\pi$-closure $\widehat W$, one may define a Riemannian metric $g_{\widehat W}^s$, for all $s>0$; and its distance function is  denoted $\widehat\rho^s$.
Note that both $g_{W}^s$ and $g_{\widehat W}^s$ are flat metrics on $M_\pi$, that, in the limit $s=0$, degenerate into positive-semidefinite symmetric $2$-tensors. Accordingly, the limits of the distance functions $\rho^s$ and $\widehat\rho^s$ are \emph{pseudo-distances} $\rho^0$ and $\widehat\rho^0$ on $M_\pi$.
Let $\varPhi\colon M_\pi \to M_\pi/\mathcal F_{\widehat W}$ be the natural projection map, and $\delta$ be the quotient metric on the leaf space $M_\pi/\mathcal F_{\widehat W}$, see Remark~\ref{thm:remmetricstructleavessoace}.

\begin{claim}\label{claim:ineqs}
For all $x,y\in M_\pi$ and $s\in(0,1]$, we have that
\begin{equation}\label{dqr}
\delta(\varPhi(x),\varPhi(y))\le\widehat\rho^s(x,y)\le\rho^s(x,y) \le\delta(\varPhi(x),\varPhi(y))+2d(s),
\end{equation}	
where $d(s)$ is as in Theorem~\ref{unifo}. Moreover,
\begin{equation}\label{drz}
\delta(\varPhi(x),\varPhi(y))\,\le\,\widehat\rho^0(x,y)\,\le\rho^0(x,y)\,\le\,\delta(\varPhi(x),\varPhi(y))
\end{equation}
or, in other words, $\delta(\varPhi(x),\varPhi(y))=\rho^0(x,y)=\widehat\rho^0(x,y)$.
\end{claim}

Note that Claim~\ref{claim:ineqs} and Lemma~\ref{ghcnv} imply that the Gromov--Hausdorff limits of both $(M_\pi,\rho^s)$ and $(M_\pi,\widehat\rho^s)$ are isometric to $\big(M_\pi/\mathcal F_{\widehat W},\delta\big)$.
Thus, to finish the proof of Theorem~\ref{mainthm:A},
replace $W$ with $\widehat W$ if necessary, and
apply Theorem~\ref{thm:crystallographiccollapse}.

We are only left with proving Claim~\ref{claim:ineqs}.
First, for all $s\in[ 0,1 ]$, we clearly have $\widehat\rho^s\le\rho^s$, while $\delta(\varPhi(x),\varPhi(y))\,\le\,\widehat\rho^s(x,y)$, which implies the two leftmost inequalities of both (\ref{dqr}) and  (\ref{drz}). To see that $\delta(\varPhi(x),\varPhi(y))\,\le\,\widehat\rho^s(x,y)$, consider any piecewise $C^1$ curve in $M_\pi$, of $\widehat \rho^s $-length $\ell_s$, joining $x$ to $y$. Lifting this curve to $\R^n$, then replacing it by its orthogonal projection onto an affine sub\-space parallel to the orthogonal complement of $\widehat W$ (which is, consequently, also orthogonal to $W$) and, finally, projecting this last curve back into $M_\pi$, we obtain a new curve joining the $\mathcal F_{\widehat W}$-leaves through $x$ and $y$, 
with $\rho^s$-length and $\widehat \rho^s$-length equal to one another and not exceeding $\ell_s$. Therefore, $\delta(\varPhi(x),\varPhi(y))\,\le\,\widehat\rho^s(x,y)$, as desired.

Second, join the $\mathcal F_{\widehat W}$-leaves through $x$ and $y$ by a 
short\-est geodesic in $M_\pi$, which hence has $\rho^1$-length $\delta(\varPhi(x),\varPhi(y))$ and is orthogonal to both leaves.
Lifted to $\R^n$, this geodesic becomes a line segment orthogonal to $\widehat W $, and hence to $W$, so that 
the $\widehat \rho^s $-length and $\rho^s$-length of the geodesic are all equal to $\delta(\varPhi(x),\varPhi(y))$. For its endpoints $x' ,y' $, with $\varPhi(x')=\varPhi(x)$ and $\varPhi(y')=\varPhi(y)$, we have that 
$\rho^s(x' ,y')\le\delta(\varPhi(x),\varPhi(y))$, and the triangle inequality gives $\rho^s(x,y)\le\rho^s(x,x')+\delta(\varPhi(x),\varPhi(y))+\rho^s(y' ,y)\le\delta(\varPhi(x),\varPhi(y))+2d(s)$. 
By Theorem~\ref{unifo},
since $d(s)\to0$ as $s\searrow0$, this implies that \eqref{dqr} and \eqref{drz} hold, completing the proof of Claim~\ref{claim:ineqs}.
\end{proof}

\begin{remark}
The collapsing deformation of a flat manifold $M_\pi$ along a subspace foliation $\mathcal F_W$ as formulated in \eqref{eq:gs} coincides with the notion of collapse of flat metrics from \cite{BetPic2016,BetDerPic2017}. Namely, the latter formulation is  in terms of a deformation of the original Bieberbach group $\pi\subset\Aff(\R^n)$ through (isomorphic) Bieberbach groups $\pi_s=\mathcal A_s\cdot \pi\cdot\mathcal A_s^{-1}\subset\Aff(\R^n)$, $s\in (0,1]$, where 
$\mathcal A_s=s\, P_W+P_{W^\perp}\in\GL(n)$, and $W\subset\R^n$ is an $H_\pi$-invariant subspace. 
Since $P_W$ and $P_{W^\perp}$ commute with $H_\pi$, the holonomy and lattice associated to $\pi_s$ are respectively $H_{\pi_s}=H_\pi$ and $L_{\pi_s}=\mathcal A_s(L_\pi)$. 
Denote by $M_{\pi_s}=\R^n/\pi_s$ the corresponding flat Riemannian manifold, that is, such that the quotient map $P_{\pi_s}\colon \R^n\to M_{\pi_s}$ is a Riemannian covering. We claim that $M_{\pi_s}$ is isometric to $(M_\pi,g_W^s)$.
Indeed, the linear isomorphism $\mathcal A_s\colon\mathds R^n\to\mathds R^n$ is equivariant with respect to the actions of $\pi$ on the domain and of $\pi_s$ on the codomain, and hence descends to a diffeomorphism 
$\widetilde{\mathcal A}_s\colon M_\pi \to M_{\pi_s}$. 
For all $z\in\mathds R^n$, $\Vert\mathrm d\widetilde{\mathcal A}_s(z)\Vert^2=s^2\Vert P_W(z)\Vert^2+\Vert P_{W^\perp}(z)\Vert^2=g^s_W(z,z)$, which means that $\widetilde{\mathcal A}_s$ is an isometry between $(M_\pi,g_W^s)$ and $M_{\pi_s}$, as claimed above.
\end{remark}

\section{Singularities of the leaf space}
\label{sec:excleaves}

In this section, we analyze different types of leaves of subspace foliations, and their relation with singularities of the leaf space, leading to the proof of Theorem~\ref{mainthm:B}.
We assume throughout that $W\subset\mathds R^n$ is an $H_\pi$-invariant $L_\pi$-generated subspace.

\subsection{Principal and exceptional leaves}\label{sub:defexceptional}
Recall that the Bieberbach group of a leaf $\mathcal F_W(u)\subset M_\pi$ is isomorphic to the subgroup $G_W(u)\subset\pi$ given by \eqref{eq:defGuW}.

\begin{definition}\label{def:exceptionalleaf}
The leaf $\mathcal F_W(u)$ is \emph{exceptional} if there exists $u'\in\mathds R^n$ and $(A,v)\in G_W(u)$ such that $(A,v)\not\in G_W(u')$, i.e., if  $G_W(u)\not\subset G_W(u')$ for some $u'\in\mathds R^n$. Leaves that are not exceptional are called \emph{principal} leaves.
\end{definition}

\begin{lemma}\label{thm:charexcept}
The leaf $\mathcal F_W(u)$ is principal if and only if $A\vert_{W^\perp}=\mathrm{Id}$ and $v\in W$ for all $(A,v)\in G_W(u)$.
\end{lemma}

\begin{proof}
Using \eqref{eq:defGuW}, it is readily seen that if $(A,v)\in\pi$ satisfies $A\vert_{W^\perp}=\mathrm{Id}$ and $v\in W$, then $(A,v)\in G_W(u')$ for all $u'\in\mathds R^n$. Thus, if all $(A,v)\in G_W(u)$ satisfy $A\vert_{W^\perp}=\mathrm{Id}$ and $v\in W$, then $\mathcal F_W(u)$ must be principal. Conversely, if $\mathcal F_W(u)$ is principal, assume $u\in W^\perp$ (otherwise replace $u$ with $u-P_W(u)$), and \eqref{eq:inclusionuu'} must hold for every $u'\in W^\perp$. In particular, setting $u'=0$ we get that $(A-\mathrm{Id})u=0$ for all $(A,v)\in G_W(u)$, which again implies $(A-\mathrm{Id})u'=0$ for all $u'\in W^\perp$. In this situation, it follows easily from \eqref{eq:defGuW} that $v\in W$ for all $(A,v)\in G_W(u)$.
\end{proof}

\begin{remark}\label{rem:samedef}
The above shows that Definition~\ref{def:exceptionalleaf} agrees with the usual notions for (regular) foliations; namely, a leaf $\mathcal F_W(u)$ is exceptional if and only if its leaf holonomy $\mathrm{Hol}_p(\mathcal F_W(u))$ is nontrivial, and principal if and only if $\mathrm{Hol}_p(\mathcal F_W(u))$ is trivial, see e.g.~\cite{molino,radeschi}.
From Remark~\ref{rem:holonomies}, the leaf holonomy $\mathrm{Hol}_p(\mathcal F_W(u))$ is the image of $\pi_1(\mathcal F_W(u),p)\cong G_W(u)$ in $\O\big(\nu_p(\mathcal F_W(u))\big)\cong\O(W^\perp)$. Thus, Lemma~\ref{thm:charexcept} states precisely that $\mathcal F_W(u)$ is principal if and only if $\mathrm{Hol}_p(\mathcal F_W(u))$ is trivial, see also \cite[Thm~10.1 (ii), (iv)]{derdzinski-piccione}.
\end{remark}

\begin{corollary}\label{thm:corHLnonexcept}
If $\mathcal F_W(u)$ is principal, then the map $H_W(u)\ni A\mapsto A\vert_W\in H_W(u)$ is injective, and $L_W(u)=L_\pi\cap W$.
\end{corollary}

The general result in foliation theory that the \emph{closest-point projection} is a covering map can be easily obtained in the context of subspace foliations as follows:

\begin{proposition}\label{thm:leafcovering}
Given $u,u'\in W^\perp$, such that $\mathcal F_W(u)$ is a principal leaf, the translation $T_{u'-u}\colon\ W+u\to W+u'$ induces a covering map $\mathcal F_W(u)\to\mathcal F_W(u')$.
\end{proposition}

\begin{proof}
The projections $P_\pi\colon W+u\to\mathcal F_W(u)$ and $P_\pi\colon W+u'\to\mathcal F_W(u')$ are covering maps, with deck transformation groups $G_W(u)$ and $G_W(u')$ respectively, see Proposition~\ref{thm:Bieberbachleaf}. Since $\mathcal F_W(u)$ is principal, $G_W(u)\subset G_W(u')$. In order to conclude, it suffices to note that for all $(A,v)\in G_W(u)$, one has $T_{u'-u}\big((A,v)x\big)=(A,v)\big(T_{u'-u}(x)\big)$ for all $x\in W+u$. This follows immediately from $A\vert_{W^\perp}=\mathrm{Id}$, see Lemma~\ref{thm:charexcept}.
\end{proof}

\begin{remark}\label{rem:isotopy}
It follows from the proof of Proposition~\ref{thm:leafcovering} that the homomorphism between fundamental groups $G_W(u)\to G_W(u')$ induced by the above covering map $\mathcal F_W(u)\to\mathcal F_W(u')$ is the inclusion.
\end{remark}

Moreover, in the realm of subspace foliations, the proof that exceptional leaves constitute a meager set is also relatively simple. Given $A\in H_\pi$, recall that the restriction $(A-\mathrm{Id})\vert_{\Ker(A-\mathrm{Id})^\perp}$ is an isomorphism, since
$\Ker(A-\mathrm{Id})^\perp=\mathrm{Im}(A-\mathrm{Id})$ by \eqref{eq:orthimker}. We denote its inverse by \[\mathcal S_A\colon\Ker(A-\mathrm{Id})^\perp\longrightarrow\Ker(A-\mathrm{Id})^\perp.\]
Define $\pi_W^\text{sing}$ to be the following subset of the Bieberbach group $\pi$:
\begin{equation}\label{eq:defpising}
\pi_W^\text{sing}=\big\{(A,v)\in\pi:A\vert_{W^\perp}\ne\mathrm{Id},\ \text{and}\ P_{W^\perp}(v)\in\Ker(A-\mathrm{Id})^\perp\big\}.
\end{equation}
It is interesting to observe that for all $u\in\mathds R^n$, if $(A,v)\in G_W(u)$ and $A\vert_{W^\perp}\ne\mathrm{Id}$,
 then $(A,v)\in\pi_W^\text{sing}$; namely:
\begin{equation*}
\begin{aligned}
(A,v)\in G_W(u) \;  &\stackrel{\eqref{eq:defGuW}}{\Longrightarrow}  \; (A-\mathrm{Id})u+v\in W\\ 
\; & \Longrightarrow \; P_{W^\perp}(v)=-P_{W^\perp}\big((A-\mathrm{Id})u\big)=-(A-\mathrm{Id})\big(P_{W^\perp}(u)\big).
\end{aligned}
\end{equation*}
Thus, we have a well-defined map:
\begin{equation*}
\pi_W^\text{sing}\ni (A,v)\longmapsto u_{(A,v)}:=\mathcal S_A\big(P_{W^\perp}(v)\big)\in\Ker(A-\mathrm{Id})^\perp.
\end{equation*}
Note that $u_{(A,v)}\in W^\perp$ for all $(A,v)\in\pi_W^\text{sing}$, since $W^\perp$ is preserved by $A-\mathrm{Id}$.

\begin{proposition}\label{thm:existenceexcept}
The set $\mathcal E_W=\big\{u\in\mathds R^n:\mathcal F_W(u)\ \text{is exceptional}\big\}$ is the union of a countable family of proper affine subspaces of $\mathds R^n$, more precisely
\begin{equation*}
\mathcal E_W=\bigcup_{(A,v)\in\pi_W^\text{sing}}\Big((A-\mathrm{Id})^{-1}(W)-u_{(A,v)}\Big).
\end{equation*}
\end{proposition}

\begin{remark}\label{thm:rempropaffsubsp}
Note that if $(A,v)\in \pi_W^\text{sing}$, then $\mathrm{Im}(A-\mathrm{Id})\cap W^\perp\ne\{0\}$, because $W^\perp$ is $A$-invariant, and $A\vert_{W^\perp}\ne\mathrm{Id}$. In particular, $\mathrm{Im}(A-\mathrm{Id})\not\subset W$, which says that the inverse image $(A-\mathrm{Id})^{-1}(W)$ is a \emph{proper} subspace of $\mathds R^n$ for all $(A,v)\in \pi_W^\text{sing}$.
\end{remark}

\begin{proof}[Proof of Proposition~\ref{thm:existenceexcept}]
Assume that $u\in(A-\mathrm{Id})^{-1}(W)-u_{(A,v)}$ for some $(A,v)\in\pi_W^\text{sing}$, i.e., $(A-\mathrm{Id})(u+u_{(A,v)})\in W$. Then:
\begin{multline*}
(A-\mathrm{Id})u+v=(A-\mathrm{Id})(u+u_{(A,v)})-(A-\mathrm{Id})u_{(A,v)}+v\\=(A-\mathrm{Id})(u+u_{(A,v)})-P_{W^\perp}(v)+v\in W,
\end{multline*}
i.e., $(A,v)\in G_W(u)$. Moreover, since $(A,v)\in\pi_W^\text{sing}$, then $A\vert_{W^\perp}\ne\mathrm{Id}$, and hence there exists $u'\in(A-\mathrm{Id})^{-1}(W^\perp\setminus\{0\})$. A direct computation shows that
\[(A-\mathrm{Id})(u+u')+v=(A-\mathrm{Id})u'+(A-\mathrm{Id})u+v=(A-\mathrm{Id})u'+P_W(v)\not\in W,\]
i.e., $(A,v)\not\in G_W(u+u')$. Therefore, $\mathcal F_W(u)$ is exceptional.

Conversely, assume $\mathcal F_W(u)$ is exceptional, and let $(A,v)\in\pi$, $u'\in\mathds R^n$ with
\begin{equation*}
(A-\mathrm{Id})u+v\in W,\quad\text{and}\quad (A-\mathrm{Id})u'+v\not\in W.
\end{equation*}
By the above, we get $P_{W^\perp}(v)=-P_{W^\perp}(A-\mathrm{Id})u$, and
\[0\ne P_{W^\perp}(A-\mathrm{Id})u'-P_{W^\perp}(v)=P_{W^\perp}(A-\mathrm{Id})(u'+u)=(A-\mathrm{Id})P_{W^\perp}(u'+u),\]
which implies that $A\vert_{W^\perp}\ne\mathrm{Id}$. Moreover:
\[P_{\Ker(A-\mathrm{Id})}\big(P_{W^\perp}(v)\big)=-P_{\Ker(A-\mathrm{Id})}\big((A-\mathrm{Id})P_{W^\perp}(u)\big)\stackrel{ \eqref{eq:orthimker}}=0,\]
i.e., $P_{W^\perp}(v)\in\Ker(A-\mathrm{Id})^\perp$, and so $(A,v)\in\pi_W^\text{sing}$. Moreover, we have that
\[P_{W^\perp}(A-\mathrm{Id})\big(u+u_{(A,v)}\big)=-P_{W^\perp}(v)+(A-\mathrm{Id})\mathcal S_A\big(P_{W^\perp}(v)\big)=0,\]
i.e., $u\in\mathcal E_W$, which concludes the proof.
\end{proof}

\subsection{Characterizing singularities}
We now describe the singularities of the leaf space $M_\pi/\mathcal F_W$, relating them with exceptional leaves of $\mathcal F_W$. Once again, although these results hold in far greater generality for totally geodesic Riemannian foliations, we provide simple and explicit proofs in the context of subspace foliations, see also \cite[Thm~10.1 (iii)]{derdzinski-piccione}.

\begin{lemma}\label{thm:nonexceptisometric}
Any two principal leaves are isometric. More generally, if $G_W(u)=G_W(u')$, then $\mathcal F_W(u)$ and $\mathcal F_W(u')$ are isometric.
\end{lemma}

\begin{proof}
Assume $u,u'\in W^\perp$, and $G_W(u)=G_W(u')$ By \eqref{eq:inclusionuu'}, $(\mathrm{Id}-A)u=(\mathrm{Id}-A)u'$, i.e., $A(u-u')=u-u'$, for all $(A,v)\in G_W(u)=G_W(u')$. This means that the isometry $(\mathrm{Id},u'-u)\colon W+u\to W+u'$ is equivariant with respect to the actions of $G_W(u)=G_W(u')$ on $W+u$ and on $W+u'$. Thus, $(\mathrm{Id},u'-u)$ induces an isometry from $\mathcal F_W(u)$ to $\mathcal F_W(u')$.
\end{proof}

A partial converse to the above statement is given as follows:

\begin{proposition}\label{thm:allisometric}
The subspace foliation $\mathcal F_W$ has no exceptional leaves if and only if all of its leaves are isometric.
\end{proposition}

\begin{proof}
By Lemma~\ref{thm:nonexceptisometric}, if $\mathcal F_W$ has no exceptional leaves, then all the leaves are isometric. Conversely, assume that $\mathcal F_W(u)$ is exceptional for some $u\in\mathds R^n$, and choose $u'\in\mathds R^n$ such that $\mathcal F_W(u')$ is principal, which is possible by Proposition~\ref{thm:existenceexcept}. 
Since $\mathcal F_W(u)$ and $\mathcal F_W(u')$ are isometric, they have the same volume. Thus, the covering map from Proposition~\ref{thm:leafcovering} is a diffeomorphism, and hence induces an isomorphism $G_W(u)\to G_W(u')$ between fundamental groups. By Remark~\ref{rem:isotopy}, this isomorphism is the inclusion $G_W(u)\subset G_W(u')$, which implies that $G_W(u)=G_W(u')$, yielding the desired contradiction.
\end{proof}

We are now in position to prove Theorem~\ref{mainthm:B} stated in the Introduction.

\begin{proof}[Proof of Theorem~\ref{mainthm:B}]
The equivalence between (ii) and (iii) is proven in Proposition~\ref{thm:allisometric}.
From Theorem~\ref{thm:crystallographiccollapse}, the leaf space $M_\pi/\mathcal F_W$ is isometric to the flat orbifold $W^\perp/\pi^\perp$.
In order to show that (i) and (iii) are equivalent, we first claim that a point in $W^\perp/\pi^\perp=M_\pi/\mathcal F_W$ is singular if and only if the corresponding leaf is exceptional. 
By definition, the singularities of $W^\perp/\pi^\perp$ correspond to orbits of the $\pi^\perp$-action on $W^\perp$ with nontrivial stabilizer.
Fix $u\in\mathcal E_W$, and choose $(A,v)\in\pi_W^\text{sing}$ such that $(A-\mathrm{Id})(u+u_{(A,v)})\in W$, see Proposition~\ref{thm:existenceexcept}. Let $x=-u_{(A,v)}\in W^\perp\cap\Ker(A-\mathrm{Id})^\perp$, so that $Ax=x-P_{W^\perp}(v)$, and hence the (nontrivial) element $\big(A\vert_{W^\perp},P_{W^\perp}(v)\big)\in\pi^\perp$ is in the stabilizer of $x$. Conversely, if $(A,v)\in\pi$, $x\in W^\perp$ are such that $\big(A\vert_{W^\perp},P_{W^\perp}(v)\big)\in\pi^\perp$ is nontrivial, $Ax+P_{W^\perp}(v)=x$, i.e., $P_{W^\perp}(v)=-(A-\mathrm{Id})x$, then clearly $P_{W^\perp}(v)\in\Ker(A-\mathrm{Id})^\perp=\mathrm{Im}(A-\mathrm{Id})$. Moreover, $A\vert_{W^\perp}\ne\mathrm{Id}$, for otherwise $P_{W^\perp}(v)=0$, contrary to the assumption that $\big(A\vert_{W^\perp},P_{W^\perp}(v)\big)$ is a nontrivial element in $\pi^\perp$. 
Therefore, $\pi_W^\text{sing}\neq\emptyset$ by \eqref{eq:defpising}, and hence $\mathcal E_W\neq\emptyset$ by Proposition~\ref{thm:existenceexcept}.
This proves the above claim, i.e., $M_\pi/\mathcal F_W$ is smooth if and only if $\mathcal F_W$ has no exceptional leaves. When this is the case, by Proposition~\ref{thm:leafcovering}, the map $M_\pi\to M_\pi/\mathcal F_W$ is a fiber bundle whose fibers are the leaves $\mathcal F_W(u)$ for any $u\in\R^n$, hence (i) and (iii) are equivalent. Finally, the equivalence between (iii) and (iv) follows from Proposition~\ref{thm:existenceexcept}, since $\mathcal F_W$ has no exceptional leaves if and only if $\pi_W^\text{sing}=\emptyset$, which is equivalent to (iv) by \eqref{eq:defpising}.
\end{proof}

\section{Existence of at least two nontrivial collapses}\label{sec:existence}

Whenever needed, we implicitly identify the rational vector space $L_\pi\tens{\mathds Z}\mathds Q$ with the $\mathds Q$-subspace of $\mathds R^n$ spanned by $L_\pi$. 
By Maschke's Theorem (see also Proposition~\ref{thm:ExistcomplHinvLgen}), the rational holonomy representation is completely reducible, so there is a decomposition of the rational vector space $L_\pi\tens{\mathds Z}\mathds Q$ of the form
\begin{equation}\label{eq:lpiq} 
L_\pi\tens{\mathds Z}\mathds Q= V_{1}^{(1)} \oplus \cdots\oplus V_{a_1}^{(1)} \oplus\ldots\oplus V_{1}^{(k)}\oplus\cdots\oplus V_{a_k}^{(k)},
\end{equation}
where the $V_j^{i}$ are pairwise distinct $\mathds Q$-irreducible $H_\pi$-invariant subspaces, with $V_j^{(i)}$ equivalent to $V_{j'}^{(i')}$ if and only if $i=i'$. Thus, the integers $a_i$ represent the multiplicity of each irreducible component, and $\widetilde V_i:=V_1^{(i)}\oplus\cdots\oplus V_{a_i}^{(i)}$ are the isotypic components of the rational holonomy representation. By Theorem~\ref{thm:lutowski}, we have that $k\geq 2$. Set $d_j=\dim(V_1^{(j)})$, for $j=1,\ldots,k$.
If the $\widetilde V_j$'s are arranged with dimensions in nondecreasing order, i.e., $d_j\le d_{j+1}$ for all $1\leq j\leq k-1$, then the i-sequence of the rational holonomy representation of $H_\pi$ is given by:
\begin{equation*}
i_\pi=\big(\!\underbrace{d_1,\ldots,d_1}_{a_1 \text{ times}},\cdots,\underbrace{d_k,\ldots,d_k}_{a_k \text{ times}}\!\big)
\end{equation*}
Let us now show that the i-sequence of the rational holonomy representation of a flat orbifold obtained by collapsing a flat manifold $M_\pi$ is a \emph{subsequence} of the i-sequence of the rational holonomy representation of $M_\pi$.

\begin{lemma}\label{thm:collapseafew}
Consider the decomposition \eqref{eq:lpiq}, and fix integers $0\leq b_j\leq a_j$,  $j=1,\ldots,k$. Let $W$
be the $H_\pi$-invariant and $L_\pi$-generated  subspace given by the real span of the rational vector subspace $V_{1}^{(1)} \oplus \cdots V_{b_1}^{(1)} \oplus\ldots\oplus V_{1}^{(k)}\oplus\cdots\oplus V_{b_k}^{(k)}$. The rational holonomy representation of the flat orbifold $M_\pi/\mathcal F_W=W^\perp/\pi^\perp$ has i-sequence given by:
\begin{equation*}
i_{\pi^\perp}=\big(\!\!\!\underbrace{d_1,\ldots,d_1}_{(a_1-b_1) \text{ times}} \!\!,\cdots,\!\underbrace{d_k,\ldots,d_k}_{(a_k-b_k) \text{ times}}\!\!\!\big).
\end{equation*}
\end{lemma}

\begin{proof}
The restriction of the orthogonal projection 
\begin{equation*}
P_{W^\perp}\colon V_{b_1+1}^{(1)} \oplus \cdots V_{a_1}^{(1)} \oplus\ldots\oplus V_{b_k+1}^{(k)}\oplus\cdots\oplus V_{a_k}^{(k)}
\longrightarrow P_{W^\perp}(L_\pi)\otimes\mathds Q
\end{equation*}
is an isomorphism of $H_\pi$-modules. Using Corollary~\ref{thm:holattcollapse}, it is easy to see that the image of each $V_i^{(j)}$ in this decomposition corresponds to an irreducible subspace of the rational holonomy representation of $M_\pi/\mathcal F_W=W^\perp/\pi^\perp$.
\end{proof}

In fact, Lemma~\ref{thm:collapseafew} also shows that any subsequence of the i-sequence of the rational holonomy representation of $M_\pi$ is the i-sequence of the rational holonomy of some collapse of $M_\pi$.
With this, we are finally ready to prove the following:

\begin{proposition}\label{thm:existemncenontrivialcollapse2}
If the i-sequence $i_\pi$ of the rational holonomy representation of $M_\pi$ is not of the form $(k,k)$, then $M_\pi$ admits at least two nontrivial collapsed limits that are not affinely equivalent.
\end{proposition}

\begin{proof}
When $i_\pi$ is not of the form $(k,k)$, then one can find two distinct and nontrivial subsequences of $i_\pi$. By Lemma~\ref{thm:collapseafew}, such subsequences correspond to nontrivial flat collapses of $M_\pi$ that are not affinely equivalent, cf.~Corollary~\ref{thm:iseqaffequivorb}.
\end{proof}

In particular, Theorem~\ref{mainthm:C} stated in the Introduction follows directly from Proposition~\ref{thm:existemncenontrivialcollapse2}, since the sum of all the elements of the i-sequence of the rational holonomy representation is equal to the dimension $n$ of the flat manifold $M_\pi$. 

\begin{remark}
Note that if the $H_\pi$-representation on $W^\perp$ is irreducible, then so is the holonomy representation of the collapsed limit $M_\pi/\mathcal F_W$, which hence is not smooth, see Theorem~\ref{thm:hissszczepa91} and Corollary~\ref{thm:holattcollapse}. In particular, this implies that the two collapsed limits in Proposition~\ref{thm:existemncenontrivialcollapse2} can be chosen to be nonsmooth flat orbifolds.
\end{remark}


\begin{thebibliography}{BDP18}

\bibitem[AR17]{alexrad}
{\sc M.~M. Alexandrino and M.~Radeschi}.
\newblock {\em Closure of singular foliations: the proof of {M}olino's
  conjecture}.
\newblock Compos. Math., 153 (2017), 2577--2590.

\bibitem[BDP18]{BetDerPic2017}
{\sc R.~G. Bettiol, A.~Derdzinski, and P.~Piccione}.
\newblock {\em Teichm\"{u}ller theory and collapse of flat manifolds}.
\newblock Ann. Mat. Pura Appl. (4), 197 (2018), 1247--1268.

\bibitem[Bie11]{bieberbach}
{\sc L.~Bieberbach}.
\newblock {\em \"{U}ber die {B}ewegungsgruppen der {E}uklidischen {R}\"aume,
  {I}}.
\newblock Math. Ann., 70 (1911), 297--336.

\bibitem[BP18]{BetPic2016}
{\sc R.~G. Bettiol and P.~Piccione}.
\newblock {\em Infinitely many solutions to the {Y}amabe problem on noncompact
  manifolds}.
\newblock Ann. Inst. Fourier (Grenoble), 68 (2018), 589--609.

\bibitem[CG86]{cheeger-gromov1}
{\sc J.~Cheeger and M.~Gromov}.
\newblock {\em Collapsing {R}iemannian manifolds while keeping their curvature
  bounded. {I}}.
\newblock J. Differential Geom., 23 (1986), 309--346.

\bibitem[CG90]{cheeger-gromov2}
{\sc J.~Cheeger and M.~Gromov}.
\newblock {\em Collapsing {R}iemannian manifolds while keeping their curvature
  bounded. {II}}.
\newblock J. Differential Geom., 32 (1990), 269--298.

\bibitem[Cha86]{charlap}
{\sc L.~S. Charlap}.
\newblock {\em Bieberbach groups and flat manifolds}.
\newblock Universitext. Springer-Verlag, New York, 1986.

\bibitem[DP]{derdzinski-piccione}
{\sc A.~Derdzinski and P.~Piccione}.
\newblock {\em Flat manifolds and reducibility}.
\newblock Preprint.
\newblock
  \htmladdnormallink{arXiv:1903.10479}{http://arxiv.org/abs/1903.10479}.

\bibitem[Fuk87]{fukaya1}
{\sc K.~Fukaya}.
\newblock {\em Collapsing {R}iemannian manifolds to ones of lower dimensions}.
\newblock J. Differential Geom., 25 (1987), 139--156.

\bibitem[Fuk88]{fukaya3}
{\sc K.~Fukaya}.
\newblock {\em A boundary of the set of the {R}iemannian manifolds with bounded
  curvatures and diameters}.
\newblock J. Differential Geom., 28 (1988), 1--21.

\bibitem[Fuk89]{fukaya2}
{\sc K.~Fukaya}.
\newblock {\em Collapsing {R}iemannian manifolds to ones with lower dimension.
  {II}}.
\newblock J. Math. Soc. Japan, 41 (1989), 333--356.

\bibitem[Fuk06]{fukaya}
{\sc K.~Fukaya}.
\newblock Metric {R}iemannian geometry.
\newblock In {\em Handbook of differential geometry. {V}ol. {II}}, 189--313.
  Elsevier/North-Holland, Amsterdam, 2006.

\bibitem[HS91]{HisSzc91}
{\sc G.~Hiss and A.~Szczepa{\'n}ski}.
\newblock {\em On torsion free crystallographic groups}.
\newblock J. Pure Appl. Algebra, 74 (1991), 39--56.

\bibitem[Lut]{lutowski18}
{\sc R.~Lutowski}.
\newblock {\em Flat manifolds with homogeneous holonomy representation}.
\newblock Preprint.
\newblock
  \htmladdnormallink{arXiv:1803.07177}{http://arxiv.org/abs/1803.07177}.

\bibitem[Mol88]{molino}
{\sc P.~Molino}.
\newblock {\em Riemannian foliations}, vol.~73 of Progress in Mathematics.
\newblock Birkh\"{a}user Boston, Inc., Boston, MA, 1988.
\newblock Translated from the French by Grant Cairns, With appendices by
  Cairns, Y. Carri\`ere, \'{E}. Ghys, E. Salem and V. Sergiescu.

\bibitem[Rad17]{radeschi}
{\sc M.~Radeschi}.
\newblock {\em Lecture notes on singular riemannian foliations}.
\newblock Unpublished notes,  (2017).
\newblock Available at:
  \htmladdnormallink{https://www.marcoradeschi.com/s/SRF-Lecture-Notes.pdf}{https://www.marcoradeschi.com/s/SRF-Lecture-Notes.pdf}.

\bibitem[Szc12]{szczepa-book}
{\sc A.~Szczepa{\'n}ski}.
\newblock {\em Geometry of crystallographic groups}, vol.~4 of Algebra and
  Discrete Mathematics.
\newblock World Scientific Publishing Co. Pte. Ltd., Hackensack, NJ, 2012.

\bibitem[Wol11]{Wolf_book}
{\sc J.~A. Wolf}.
\newblock {\em Spaces of constant curvature}.
\newblock AMS Chelsea Publishing, Providence, RI, sixth edition, 2011.

\end{thebibliography}
\end{document}